\documentclass[%
 aip,
 amsmath,amssymb,
 reprint,%
]{revtex4-1}

\usepackage{graphicx}
\usepackage{dcolumn}
\usepackage{bm}
\usepackage{amsmath,amsfonts,amssymb,amsthm}
\usepackage{color,indentfirst,graphicx}
\usepackage[utf8]{inputenc}
\usepackage[T1]{fontenc}
\usepackage{mathptmx}
\usepackage{etoolbox}

\makeatletter
\def\@email#1#2{%
 \endgroup
 \patchcmd{\titleblock@produce}
  {\frontmatter@RRAPformat}
  {\frontmatter@RRAPformat{\produce@RRAP{*#1\href{mailto:#2}{#2}}}\frontmatter@RRAPformat}
  {}{}
}%
\makeatother

\newtheorem{proposition}{Proposition}

\renewcommand{\leq}{\leqslant}
\renewcommand{\geq}{\geqslant}
\def\*#1{\mathbf{#1}}
\DeclareMathOperator{\dv}{div}
\def\ve{\varepsilon}
\def\vk{\varkappa}
\def\vp{\varphi}

\begin{document}

\preprint{AIP/123-QED}

\title[Doubly Reduced Model for Dynamics of Heterogeneous Mixtures]{On a Doubly Reduced Model for Dynamics of Heterogeneous Mixtures\\of Stiffened Gases, its Regularizations and their Implementations}
\author{A. Zlotnik}
\altaffiliation[Also at ]{Keldysh Institute of Applied Mathematics, Miusskaya Sqr., 4, 125047 Moscow, Russia}
\author{T. Lomonosov}%
\altaffiliation[Also at ]{Keldysh Institute of Applied Mathematics, Miusskaya Sqr., 4, 125047 Moscow, Russia}
\email{azlotnik@hse.ru, tlomonosov@hse.ru}
\affiliation{
$^1$Higher School of Economics University, Pokrovskii Bd. 11, 109028 Moscow, Russia\\
}
%

\date{\today}

\begin{abstract}
We deal with the reduced four-equation model for the dynamics of heterogeneous compressible binary mixtures with the stiffened gas equations of state.
We study its further reduced form, with the excluded volume concentrations, and with a quadratic equation for the common pressure of the components; this form can be called a quasi-homogeneous  form.
We prove new properties of the equation, derive simple formulas for the squared speed of sound and present an alternative proof for a formula that relates it to the squared Wood speed of sound; also, a short derivation of the pressure balance equation is given.
For the first time, we introduce regularizations of the heterogeneous model (in the quasi-homogeneous form).
Previously, regularizations of such type were  developed only for the homogeneous mixtures of perfect polytropic gases, and it was unclear how to cover the case considered here.
In the 1D case, based on these regularizations, we construct new explicit two-level in time and symmetric three-point in space finite-difference schemes without limiters, and provide numerical results for various flows with shock waves.
\end{abstract}
\maketitle

\begin{quotation}
Dynamic problems for the heterogeneous binary mixtures of compressible gases and fluids are of great theoretical and practical interest.
For the purpose of their mathematical description, various models containing from four to seven partial differential equations were developed.
The most reduced of them is the four-equation model that describes one-velocity and one-temperature flows in which both velocity and temperature of all the components are the same, and the components also have a common pressure. The model has various important applications.
In the case of the widely used stiffened gas equations of state, this model was rather recently further reduced to contain the minimal amount (four) of the sought functions.
This doubly reduced model is especially convenient for the purpose of constructing numerical methods for computer simulation of flows.
In this paper, we answer some theoretical questions that arise in this model including the choice of physically correct pressure value, derivation of the compact formula for the speed of sound and its comparison with alternative formulas.
For the first time for this kind of models, we also introduce regularizations of those types that are well developed in the cases of the single-component gas and homogeneous mixtures.
This allows us to construct rather simple explicit two-level in time and symmetric three-point in space finite-difference schemes without limiters in the 1D case.
We confirm the efficiency of the approach by computer simulations of various mixture flows with shock waves.
\end{quotation}

\section{\label{sec:level1}Introduction}
A hierarchy of models was developed for dynamics of the heterogeneous binary or multicomponent mixtures of compressible gases and fluids, see \cite{FMM10,FL11,ZMWS22} and the references therein.
The most reduced of them is the four-equation model for binary mixtures (i.e., the model that contains four partial differential equations (PDEs)) for one-velocity and one-temperature flows in which both velocity and temperature of all the components are the same, and the components also have a common pressure.
Its further reduced form, with the excluded volume concentrations, and with a quadratic equation for the common pressure of the components, was suggested in \cite{LeMSN14} in the case of the widely used stiffened gas equations of state and can be called a quasi-homogeneous form.
The doubly reduced model is especially convenient for constructing new numerical methods for computer simulation of complicated flows with phase transitions, first see \cite{LeMSN14}.
Further development of such numerical approach was accomplished for binary mixtures in\cite{CBS17IJNMF} and multicomponent mixtures in\cite{CBS17CF}. An example with both components described by the stiffened gas equations of state was considered in\cite{ABR20}.
In the frame of this approach, recently the non-conservative residual distribution scheme was suggested and tested in\cite{BCPCA22}, a pressure-based diffuse interface method for low-Mach multiphase flows with mass transfer was developed in\cite{DSPB22} and a numerical relaxation techniques with the enlarged capabilities to describe heat and mass transfer processes was given in\cite{P22}.
A brief review of alternative methods for mixtures, with extended references, can be found in \cite{ZMWS22}.

\par  In Section II
of this paper, we first present and study the reduced four-equation heterogeneous model and
its further reduction to the quasi-homogeneous form.
We prove new properties of the equation for the common pressure including the correct choice of its physical root, derive two rather simple formulas for the squared speed of sound (with two different derivations for the main of them) and the balance PDE for the pressure.
We also give an alternative proof for a formula that relates the squared speed of sound to the well-known Wood one.
We also compare the derived formula with two other known expressions.
Recall that the speed of sound is used for both constructing various numerical methods for the problems in question and choosing the time step that guarantees stability of explicit methods.

\par In Section III,
for the first time, we construct the so-called quasi-gasdynamic (QGD) regularization and quasi-hydrodynamic (QHD) one (essentially, a simplified QGD regularization) for the heterogeneous model in the quasi-homogeneous form.
The regularizations of such type are well-developed and are applied to a number of practical problems for the single-component gas, see \cite{Ch04,E07} and a lot of subsequent papers.
Their extension to homogeneous binary mixtures of perfect polytropic gases was initiated in \cite{E07}.
For two-velocity and two-temperature binary mixtures, the approach was improved theoretically in\cite{EZCh14} and applied practically, in particular, in\cite{KKPP18}.
The QHD regularization for mixtures with the phase interactions was constructed in\cite{BS18} and practically implemented, in particular, in\cite{BZ21}.
Several regularizations and their discretizations for one-velocity and one-temperature homogeneous mixtures were sequentially constructed and tested  in\cite{EZSh19,ZFL22} and later, with different regularizing velocities,  in\cite{ESh22} and, also taking into account diffusion fluxes, in\cite{ZL23}.
In the latter case, some theoretical aspects have recently been studied in \cite{ZF22MMAS,ZL23}; they include the validity of the entropy balance PDEs with non-negative energy productions, the Petrovsky parabolicity of the system and $L^2$-dissipativity of its linearized version.
We emphasize that, for perfect polytropic gases, the four-equation homogeneous and heterogeneous models are equivalent. But this is not the case for the stiffened gases, and attempts to apply some simple modifications of the homogeneous model to the heterogeneous case mostly fail except for some particular cases\cite{ESh22}. Recall that the components occupy their own volume in the heterogeneous models and the same volume in the homogeneous ones.

To construct the first successful QGD regularization for the heterogeneous mixtures of the stiffened gases, we apply a procedure from \cite{Z12MM2,ZL23} to the above doubly reduced model.
For both the QGD- and QHD-regularizations, we provide the additional balance PDEs for the mass, kinetic and internal energies of the mixture. Moreover, in the QHD case, we derive the balance PDE for the mixture entropy with the non-negative entropy production.
Notice that, for the single-component gas,  other regularizations are also used for constructing numerical methods\cite{GPT16,FL_MM20,DS21}.

\par In Section IV,
to verify the constructed regularized systems of PDEs at least in the 1D case, we construct explicit two-level in time and symmetric three-point in space finite-difference schemes without limiters which are conservative in the mass of components and the momentum and total energy of the mixture.
We also derive the discrete balance equations for the mass, kinetic and internal energies of the mixture using the technique from \cite{Z12CMMP}.
Finally, Section V
is devoted to numerical experiments that are based on the constructed schemes.
We implement various known tests that concern flows in shock tubes from papers \cite{KLC14,LF11,CBS17IJNMF,YC13,LA12,ABR20}.
Also Appendix contains the proofs of all Propositions from Section II.

\section{A reduced system of PDEs for the dynamics of heterogeneous mixtures of stiffened gases and its further reduction}
\label{Section II}
The reduced four-equation system of PDEs for the heterogeneous one-velocity and one-temperature compressible binary mixture consists of the balance PDEs for the mass of components, total momentum and total energy
\begin{eqnarray}
\partial_t(\alpha_kr_k)+\dv(\alpha_kr_k\*u)=0,\ \ k=1,2,
\label{mass eq alpha heter}\\[1mm]
\partial_t(\rho\*u)+\dv(\rho\*u\otimes\*u)+\nabla p=\dv\Pi^{NS}+\rho\*f,
\label{moment eq heter}\\[1mm]
\partial_t(\tfrac12\rho|\*u|^2+\rho\ve)+\dv\big((\tfrac12\rho|\*u|^2+\rho\ve+p)\*u\big)
\nonumber\\[1mm]
=\dv(-\*q^F+\Pi^{NS}\*u)+\rho\*u\cdot\*f+Q,
\label{energy eq heter}
\end{eqnarray}
for example, see \cite{LeMSN14} in the case $\Pi^{NS}=0$, $\*f=0$ and $Q=0$.
Here the main sought functions are the density $r_k>0$ and the volume fraction $0<\alpha_k<1$ of the heterogeneous component, $k=1,2$, the common velocity $\*u$ and absolute temperature $\theta>0$ of the mixture.
These functions depend on $x=(x_1,\ldots,x_n)\in\Omega$ and $t\geq 0$, where $\Omega$ is a domain in $\mathbb{R}^n$, $n=1,2,3$.
Hereafter vector-functions are written in bold,
and the operators $\dv=\nabla\cdot$, $\nabla=(\partial_1,\ldots,\partial_n)$, $\partial_t=\partial/\partial t$ and $\partial_i=\partial/\partial x_i$ are involved.
The symbols $\otimes$ and $\cdot$ correspond to the tensor and scalar products of vectors, the tensor divergence is taken with respect to its first index, and, below,
$\langle\cdot\rangle$ means the summation over index $k=1,2$.

\par The following additional relations are used
\begin{eqnarray}
 \langle\alpha_k\rangle=1,\ \
 \rho=\langle\alpha_kr_k\rangle,\ \ \rho\ve=\langle\alpha_kr_k\ve_k(r_k,\theta)\rangle,
\label{rho rhoeps}\\[1mm]
 p=p_1(r_1,\theta)=p_2(r_2,\theta)>0,\ \
\label{eq for pk}
\end{eqnarray}
where $p_k(r_k,\theta)$ and $\ve_k=\ve_k(r_k,\theta)$ are the pressure and specific internal energy
of the $k$th component ($k=1,2$),
$\rho$ and $\ve$ are the density and specific internal energy of the mixture, and
$p$ is the common pressure of the components.
In particular, Eq. \eqref{eq for pk} means that the pressures $p_k$ of the components are equal to each other, and this is the additional algebraic equation to PDEs
\eqref{mass eq alpha heter}-\eqref{energy eq heter} and formula $\langle\alpha_k\rangle=1$
that is
required to define all the sought functions listed above.

More specifically, we apply \textit{the stiffened gas equations of state} in its 
well-known form
\begin{eqnarray}
 p_k(r_k,\theta)=R_kr_k\theta-p_{*k},\ \ \ve_k(r_k,\theta)=c_{Vk}\theta+\frac{p_{*k}}{r_k}+\ve_{0k},
\label{sg EOS}
\end{eqnarray}
where $R_k>0$, $c_{Vk}>0$, $p_{*k}\geq 0$ and $\ve_{0k}$ are given physical constants, $k=1,2$.
In addition, $R_k=(\gamma_k-1)c_{Vk}$, where $\gamma_k>1$ is the adiabatic exponent, and
let $c_{pk}=\gamma_k c_{Vk}$. Recall that the perfect polytropic case corresponds to $p_{*k}=\ve_{0k}=0$.

\par The classical Navier-Stokes viscosity tensor and the Fourier heat flux are given by the formulas
\begin{eqnarray}
 \begin{array}{c}\Pi^{NS}=\mu\big(\nabla\*u+(\nabla\*u)^T\big)+\big(\lambda-\tfrac23\mu\big)(\dv\*u)\mathbb{I},\\[1mm]
 -\*q^F=\vk\nabla\theta,
 \end{array}
\label{Pi NS q F}
\end{eqnarray}
where $\mu\geq 0$, $\lambda\geq 0$ and $\vk\geq 0$ are the total viscosity and heat conductivity coefficients
(which 
may depend on the sought functions),
$\nabla\*u=\{\partial_iu_j\}_{i,j=1}^n$ and $\mathbb{I}$ is the $n$-th order unit tensor. 
For $\mu=\lambda=0$ and $\vk=0$, these terms vanish.
Also $\*f$ and $Q\geq 0$ are the given density of body forces and intensity of the heat sources. In comparison with \cite{LeMSN14}, we omit the phase transfer terms here but add the Navier-Stokes ones.

\par We define the alternative density $\rho_k=\alpha_kr_k$ of the $k$-th component.
Equations of state \eqref{sg EOS} imply sequentially
\begin{eqnarray}
 \alpha_k(p_k+p_{*k})=R_k\rho_k\theta,
\label{sg EOS cons1}\\[1mm]
 \rho_k(\ve_k-\ve_{0k})=c_{Vk}\rho_k\theta+\alpha_kp_{*k}=c_{Vk}\rho_k\theta+\frac{R_k\rho_kp_{*k}}{p_k+p_{*k}}\theta.
\label{sg EOS cons2}
\end{eqnarray}
Using the equations $\langle\alpha_k\rangle=1$ and $p=p_k$, we get the formulas
\begin{eqnarray}
 p=R\rho\theta-\langle\alpha_kp_{*k}\rangle,
\label{sg EOS cons1a}\\[1mm]
 \rho(\ve-\ve_0)=\langle\rho_k(\ve_k-\ve_{0k})\rangle=c_V\rho\theta+\langle\alpha_kp_{*k}\rangle,
\label{sg EOS cons2a}\\[1mm]
 \rho(\ve-\ve_0)+p=\gamma c_V\rho\theta
\label{sg EOS cons2b}
\end{eqnarray}
that contain the functions-coefficients of the mixture such that
\begin{equation}
 \left.\rho\ve_0=\langle\rho_k\ve_{0k}\rangle,\,
 \rho R=\langle R_k\rho_k\rangle,\,
 \rho c_V=\langle c_{Vk}\rho_k\rangle,\,
 \gamma=\frac{R}{c_V}+1.\right.
\label{coeff of mixture}
\end{equation}

\par From equality \eqref{sg EOS cons1}, we find
\[
 \frac{\alpha_1}{\alpha_2}=g(p)\frac{\rho_1}{\rho_2},\ \ g(p):=\frac{p+p_{*2}}{p+p_{*1}}\frac{R_1}{R_2}.
\]
Thus, the volume fractions $\alpha_k$ can be expressed in terms of the corresponding mass ones $y_k=\rho_k/\rho$:
\begin{eqnarray}
 \alpha_1=\frac{g(p)y_1}{g(p)y_1+1-y_1},\ \ \alpha_2=\frac{y_2}{g(p)(1-y_2)+y_2}.
\label{vol fr mass fr}
\end{eqnarray}

\par Equalities \eqref{sg EOS cons1}-\eqref{sg EOS cons2} lead to the relations
\begin{eqnarray}
 \Big\langle\frac{R_k\rho_k}{p+p_{*k}}\Big\rangle\theta=1,
\label{sg EOS cons3}\\[1mm]
 \rho(\ve-\ve_0)
 =\Big(\rho c_V+\Big\langle\frac{R_k\rho_kp_{*k}}{p+p_{*k}}\Big\rangle\Big)\theta.
\label{sg EOS cons4}
\end{eqnarray}

\par Expressing $\theta$ from Eq. \eqref{sg EOS cons3}, inserting it in Eq. \eqref{sg EOS cons4} and dividing the result by $\rho c_V$, we derive the following rational equation for $p$ in dependence on $\rho_1,\rho_2$ and $\ve$:
\begin{eqnarray}
\Big\langle\frac{\sigma^{(k)}(\rho(\ve-\ve_0)-p_{*k})}{p+p_{*k}}\Big\rangle=1.
\label{orig eq for p}
\end{eqnarray}
Here the following relations hold
\begin{equation}
 \begin{array}{cc}
\displaystyle{\rho=\langle\rho_k\rangle,\ \ \sigma^{(k)}=\sigma^{(k)}(\rho_1,\rho_2)=\frac{R_k\rho_k}{c_V\rho}>0,}
\\[1mm]
\displaystyle{\langle\sigma^{(k)}\rangle=\frac{R}{c_V}
 =\gamma-1.}
 \end{array}
\label{rho sigmak}
\end{equation}

\par This rational equation is reduced to the quadratic equation
\begin{eqnarray}
 p^2-bp-c=0,
\label{quad eq for p}
\end{eqnarray}
with the coefficients
\begin{eqnarray}
 b=\langle\sigma^{(k)}(\rho(\ve-\ve_0)-p_{*k})-p_{*k}\rangle,
\label{coeff b}\\[1mm]
 c=\sigma^{(1)}(\rho(\ve-\ve_0)-p_{*1})p_{*2}
 \nonumber\\[1mm]
 +\sigma^{(2)}(\rho(\ve-\ve_0)-p_{*2})p_{*1}-p_{*1}p_{*2}
 \nonumber\\[1mm]
 =(\sigma^{(1)}p_{*2}+\sigma^{(2)}p_{*1})\rho(\ve-\ve_0)-\gamma p_{*1}p_{*2}.
\label{coeff c}
\end{eqnarray}
Let $d:=b^2+4c$ be its discriminant.
For $d>0$,  the quadratic Eq. \eqref{quad eq for p} has the roots
\begin{eqnarray} p_\pm=p_\pm(\rho_1,\rho_2,\rho\ve)
=\tfrac12(b\pm\sqrt{d}),\ \ p_-<p_+.
\label{p pm roots}
\end{eqnarray}
But for $p_{*1}p_{*2}\neq 0$ (this case arises in some applications, for example, see test G below), the property $d>0$ and the correct choice of the physical root are not obvious and are analyzed below.

\par Note that the transition from Eq. \eqref{orig eq for p} to \eqref{quad eq for p} is not completely equivalent.
For example, in the case $p_{*1}=p_{*2}=p_*$, the unique root of the first equation is $p=R\rho\theta-p_*$, but the second one has
an additional parasitic root $p=-p_*$.
Also, in the limit case where $\alpha_k=1$ and $\alpha_l=0$ (if $l\neq k$) at some point $(x,t)$, we have $\sigma^{(k)}=\gamma_k-1$ and $\sigma^{(l)}=0$, thus, Eq. \eqref{orig eq for p} for $p$ is reduced to
\begin{eqnarray*}
 p+p_{*k}=\sigma^{(k)}(\rho(\ve-\ve_0)-p_{*k})\nonumber\\[1mm]
 =(\gamma_k-1)(\rho_k(\ve_k-\ve_{0k})-p_{*k})=(\gamma_k-1)c_{Vk}\rho_k\theta,
\end{eqnarray*}
i.e., $p=R_k\rho_k\theta-p_{*k}$ that is natural.
In this case, the quadratic Eq. \eqref{quad eq for p} has the additional parasitic root $p=-p_{*l}<0$.
\begin{proposition}
\label{the main}
Let $\Delta_*:=p_{*2}-p_{*1}$.
The following formulas hold
\begin{eqnarray}
 b=p_+ +p_-,\ \ c=-p_+p_-\geq 0,
\label{form for b and c}
\end{eqnarray}
where
\begin{eqnarray}
\begin{array}{r}
 p_+=R\rho\theta-\langle\alpha_kp_{*k}\rangle>0,\\[1mm] p_-=-\Big(\alpha_1p_{*2}+\alpha_2p_{*1}+\frac{\alpha_1\alpha_2}{c_V\rho\theta}\Delta_*^2\Big)\leq 0.
 \end{array}
\label{form for pp and pm}
\end{eqnarray}

Consequently, $d>0$, thus, these $p_\pm$ and those given by formula \eqref{p pm roots} are the same.
\end{proposition}

\par Recall that the proofs of all Propositions in this Section are put in Appendix.

\par This Proposition guarantees that $p_+$ is the physical root and $p_-$ is the parasitic one.
Notice that the found formula for $p_-$ is also of interest since it allows to prove additional results, see Propositions \ref{prop: two speeds of sound} and \ref{prop: 3 sq speeds of sound} below.
\begin{proposition}
\label{lem: d greater 0}
The following formula holds
\begin{eqnarray}
d=(b_1-b_2)^2+4a_1a_2
\label{second form for d}
\end{eqnarray}
with $b_k=a_k-p_{*k}$ and $a_k=\sigma^{(k)}(\rho(\ve-\ve_0)-p_{*k})$, $k=1,2$, see\cite{LeMSN14}, and also
\begin{eqnarray}
d=\big[(\alpha_2\sigma^{(1)}-\alpha_1\sigma^{(2)})^2+2(\alpha_1\sigma^{(2)}+\alpha_2\sigma^{(1)})+1\big]\Delta_*^2
\nonumber\\[1mm]
+2c_V\rho\theta\big[(\alpha_2\sigma^{(1)}-\alpha_1\sigma^{(2)})(\gamma-1)+\sigma^{(1)}-\sigma^{(2)}\big]\Delta_*
\nonumber\\[1mm]
+((\gamma-1)c_V\rho\theta)^2>0,
\label{third form for d}
\end{eqnarray}
where $d$ is represented as
a quadratic polynomial with respect to $\Delta_*=p_{*2}-p_{*1}$.
\end{proposition}

\par Note that we have
$\alpha_2\sigma^{(1)}-\alpha_1\sigma^{(2)}=-\alpha_1\alpha_2\Delta_*/(c_V\rho\theta)$ in \eqref{third form for d} (due to formula \eqref{form for sigma k}, see below).

\par Also $d_0=0$ is equivalent to $\rho_1\rho_2=0$.
For example, if $\rho_1=0$, then
$\gamma-1=\sigma^{(2)}=\gamma_2-1$, $d=[(\alpha_1\sigma^{(2)}+1)\Delta_*-\sigma^{(2)}c_V\rho\theta]^2$ and $\sigma^{(2)}c_V\rho=R_2\rho_2$, thus $d=0$ means that $R_2\rho_2\theta=(\alpha_1(\gamma_2-1)+1)\Delta_*$; the latter is impossible for $\alpha_1=0$ and $p_2>0$.

\par The additional balance PDEs for the mass, kinetic and internal energies of the mixture
\begin{eqnarray}
 \partial_t\rho + \dv(\rho\*u)=0,
\label{sum mass homog}\\[1mm] \tfrac12\partial_t(\rho|\*u|^2)+\tfrac12\dv\big(\rho|\*u|^2\*u\big)+\*u\cdot\nabla p\nonumber\\[1mm]
=(\dv\Pi^{NS})\cdot\*u+\rho\*f\cdot\*u,
\label{kin en homog}\\[1mm]
\partial_t(\rho\ve)+\dv(\rho\ve\*u)+p\dv\*u
\nonumber\\[1mm]
=\dv(-\*q^F)+\Pi^{NS}:\nabla\*u+Q
\label{int en homog}
\end{eqnarray}
are sequentially derived in a standard manner.
Here $:$ denotes the scalar product of tensors.
In particular, Eq. \eqref{sum mass homog} arises by applying $\langle\cdot\rangle$ to Eqs. \eqref{mass eq alpha heter}.
\begin{proposition}
\label{prop:speed of sound}
The following formula for the squared speed of sound and the balance PDE for $p_+$ hold
\begin{eqnarray}
c_s^2:=\partial_\rho p_++\frac{p_+}{\rho^2}\partial_\ve p_+\nonumber\\[1mm]
=\frac{\gamma(p_+ +p_{*1})(p_+ +p_{*2})}{\rho\sqrt{d}}>0,
\label{for speed of sound}\\
\partial_tp_++\*u\cdot\nabla p_+ +\rho c_s^2\dv\*u\nonumber\\[1mm]
 =\frac{c_s^2}{\gamma c_V\theta}\big(\dv(-\*q^F)+\Pi^{NS}:\nabla\*u+Q\big),
\label{eq for p}
\end{eqnarray}
where the derivatives $\partial_\rho$ and $\partial_\ve$ are taken
in assumption that $\ve_0$ and $\sigma^{(k)}$, $k=1,2$, are constant in \eqref{coeff b}-\eqref{coeff c} following \cite{FMM10,LeMSN14}.
\end{proposition}

\par Formula \eqref{for speed of sound} is much more compact than the original one given in \cite{LeMSN14}.

\par We also give the representation of the differential of $p_+$ and, consequently, another derivation of the formula for $c_s^2$ following \cite[Section 3.2.3 and Proposition 11]{FMM10} and \cite{A99}.
\begin{proposition}
\label{prop:speed of sound A}
The differential of $p_+$ can be expressed in terms of the differentials of $\rho_1$, $\rho_2$ and $\rho\ve$ as follows
\begin{eqnarray}
dp_+=\langle\mathcal{P}_kd\rho_k\rangle
+\mathcal{P} d(\rho\ve),
\label{form for dp}
\end{eqnarray}
where the functions $\mathcal{P}_k$ and $\mathcal{P}$ are given by the formulas
\begin{eqnarray*}
\sqrt{d}\mathcal{P}_k
=(-1)^k\mathcal{R}\rho_k-\sqrt{d}\mathcal{P}\ve_{0k},
\\
\mathcal{R}=\frac{R_1c_{V2}\mathcal{H}_1-R_2c_{V1}\mathcal{H}_2}{(c_V\rho)^2},
\\
\mathcal{H}_k=(\rho(\ve-\ve_0)-p_{*k})(p_+ +p_{*(3-k)}),
\\
\sqrt{d}\mathcal{P}=(\gamma-1)p_+
+\sigma^{(1)}p_{*2}+\sigma^{(2)}p_{*1},
\end{eqnarray*}
with $k=1,2$ and $d$ being the discriminant of Eq. \eqref{quad eq for p}.
\par Consequently, the following
Abgrall-type formula for the squared speed of sound holds
\begin{eqnarray*}
c_s^2=\Big\langle\frac{\rho_k}{\rho}\mathcal{P}_k\Big\rangle
+\frac{\rho\ve+p_+}{\rho}\mathcal{P}
=\frac{\gamma(p_+ +p_{*1})(p_+ +p_{*2})}{\rho\sqrt{d}}.
\end{eqnarray*}
\end{proposition}

\par Formula \eqref{for speed of sound} for $c_s^2$ has originally been  derived from the quadratic equation \eqref{quad eq for p}.
Interestingly, a similar technique applied to the rational equation \eqref{orig eq for p} leads to another formula for $c_s^2$ without radicals.
\begin{proposition}
\label{prop: 2nd formula for c_s^2}
The following alternative formula for $c_s^2$ holds
\begin{eqnarray}
c_s^2=\frac{\gamma}{\rho}
\Big\langle\frac{\sigma^{(k)}(\rho(\ve-\ve_0)-p_{*k})}{(p_+ +p_{*k})^2}\Big\rangle^{-1}.
\label{for speed of sound 2}
\end{eqnarray}
\end{proposition}

\par It may seem strange how two such different given formulas for $c_s^2$ can coincide, but we will demonstrate that explicitly in Appendix after the proof of Proposition \ref{prop: 2nd formula for c_s^2}.

\par Let us compare some definitions of the squared speed of sound in mixtures.
\begin{proposition}
\label{prop: two speeds of sound}
The following formula relating $c_s^2$ and $c_{sW}^2$ holds
\begin{eqnarray}
\frac{1}{\rho c_s^2}=\frac{1}{\rho c_{sW}^2}
+\frac{c_{p1}\alpha_1r_1 c_{p2}\alpha_2r_2}{\theta\rho c_p}(\zeta_1-\zeta_2)^2,
\label{cs cw frac}
\end{eqnarray}
where $c_{sW}^2$ is the well-known squared Wood speed of sound in mixtures such that
\begin{eqnarray*}
 \frac{1}{\rho c_{sW}^2}=\Big\langle\frac{\alpha_k}{r_kc_{sk}^2}\Big\rangle,\ \ c_{sk}^2:=\gamma_k(\gamma_k-1)c_{Vk}\theta,
\end{eqnarray*}
with $\rho c_p=\langle c_{pk}\alpha_kr_k\rangle$ and
\begin{eqnarray}
 \zeta_k:=\Big(1-\frac{1}{c_{pk}}\partial_\theta\ve_k(\theta,p_k)\Big)\frac{\theta}{p_k}=\frac{1}{c_{pk}r_k},\ \
 k=1,2.
\label{cp zeta}
\end{eqnarray}

Consequently, we have $ c_s^2\leq c_{sW}^2$.
\end{proposition}

\par Applying formula \eqref{for speed of sound} for $c_s^2$, we can also compare the three formulas for the squared speed of sound in mixtures known in the literature.
\begin{proposition}
\label{prop: 3 sq speeds of sound}
The inequalities hold
\begin{eqnarray}
 c_s^2\leq\gamma(\gamma-1)c_V\theta\leq\Big\langle\frac{\rho_k}{\rho}c_{sk}^2\Big\rangle
 =\frac{1}{\rho}\langle\alpha_k\gamma_k(p_k+p_{*k})\rangle.
\label{ineq for cs}
\end{eqnarray}
\end{proposition}

\par According to the proof given in Appendix, the first inequality \eqref{ineq for cs} can turn into the equality only in the cases $p_{*1}=p_{*2}$, or $\alpha_1=0$, or $\alpha_2=0$.

\par The results can be generalized to the case of multicomponent mixtures provided that $p_{*k}$ take only two distinct values similarly to \cite{CBS17CF}.
Also they are partially extended to the case of the more general Noble-Abel stiffened-gas equations of state \cite{LeMS16,SBLeM16} that is presented in another paper \cite{ZL23DM};
in this case, the proofs become more cumbersome.

\par The balance PDEs for the mass of components, total momentum and total energy are as follows
\begin{eqnarray}
\partial_t\rho_k+\dv(\rho_k\*u)=0,\ \ k=1,2,
\label{mass eq alpha NS}\\[1mm]
\partial_t(\rho\*u)+\dv(\rho\*u\otimes\*u)+\nabla p=\dv\Pi^{NS}+\rho\*f,
\label{moment eq NS}\\[1mm]
\partial_t\big(\tfrac12\rho|\*u|^2+\rho\ve\big)+\dv\big((\tfrac12\rho|\*u|^2+\rho\ve+p)\*u\big)
\nonumber\\[1mm]
=\dv(-\*q^F+\Pi^{NS}\*u)+\rho\*u\cdot\*f+Q,
\label{energy eq NS}
\end{eqnarray}
see \cite{LeMSN14} in the case $\Pi^{NS}=0$ and $Q=0$.
Here the main sought functions are the alternative densities
$\rho_k>0$, $k=1,\,2$, the velocity $\*u$ and the specific internal energy $\ve$ of the mixture.
Also $\rho=\langle\rho_k\rangle$, but formulas \eqref{rho rhoeps} and \eqref{sg EOS} are not in use.
The pressure $p$ and temperature $\theta$ are given by the formulas
\begin{equation}
p(\rho_1,\rho_2,\ve)=p_+=\frac12(b+\sqrt{d}),\, \theta(\rho_1,\rho_2,\ve)=\frac{\rho(\ve-\ve_0)+p}{\gamma c_V\rho},
\label{p and temp}
\end{equation}
see formulas \eqref{p pm roots} and \eqref{sg EOS cons2b}.
Recall that here $d=b^2+4c$ (alternatively, formula \eqref{second form for d} can be used),
with $b=b(\rho_1,\rho_2,\ve)$ and $c=c(\rho_1,\rho_2,\ve)$ given in definitions
\eqref{coeff b}, \eqref{coeff c} and \eqref{rho sigmak}.

\par We emphasize that this system does not contain $\alpha_k$ and $r_k=\rho_k/\alpha_k$, $k=1,2$, although they can be computed a posteriori,
in particular, see \eqref{vol fr mass fr}, or, according to \eqref{sg EOS cons1}, we have
\begin{eqnarray}
 \alpha_k=\frac{R_k\rho_k\theta}{p_+ +p_{*k}},\ \ k=1,2.
\label{form for alpha k}
\end{eqnarray}
Recall that this formula
and the property $\langle\alpha_k\rangle=1$ imply an alternative formula for $\theta$:
\begin{eqnarray}
\theta=\Big\langle\frac{R_k\rho_k}{p_+ +p_{*k}}\Big\rangle^{-1},
\label{2nd form for temp}
\end{eqnarray}
that we apply in our computations below.
For computing $r_k$, the formula $r_k=(p_+ +p_{*k})/(R_k\theta)$ seems to be more reliable.

\par Notice that the quasi-homogeneous form \textit{is equivalent} to the original heterogeneous one.
Indeed, formulas \eqref{form for alpha k} and \eqref{2nd form for temp} imply that
\begin{equation}
 \langle\alpha_k\rangle=1,\,
 p_k=R_kr_k\theta-p_{*k}=R_k\frac{\rho_k}{\alpha_k}\theta-p_{*k}=p_+,\, k=1,2,
\label{p k and p +}
\end{equation}
see the first equation of state \eqref{sg EOS}, and lead to Eqs. \eqref{eq for pk}.
Next, we have
\begin{eqnarray}
 \langle \alpha_kr_k\ve_k\rangle
 =\Big\langle \alpha_kr_k\Big(c_{Vk}\theta+\frac{p_{*k}}{r_k}+\ve_{0k}\Big)\Big\rangle
\nonumber\\[1mm]
 =c_V\rho\theta+\langle \alpha_kp_{*k}\rangle+\rho\ve_0,
\label{sum of ve k}
\end{eqnarray}
see the second equation of state \eqref{sg EOS}.
The quadratic Eq. \eqref{quad eq for p} implies the rational Eq. \eqref{orig eq for p}. Due to formulas \eqref{form for alpha k} and \eqref{2nd form for temp}, the latter equation can be rewritten as
\[
 \frac{1}{c_V\rho\theta}\rho(\ve-\ve_0)-\frac{\langle \alpha_kp_{*k}\rangle}{c_V\rho\theta}=1.
\]
Thus, $\rho(\ve-\ve_0)=c_V\rho\theta+\langle \alpha_kp_{*k}\rangle$, and formula \eqref{sum of ve k} implies that $\langle \alpha_kr_k\ve_k\rangle=\rho\ve$, i.e., it implies the third Eq. \eqref{rho rhoeps}.
Also note that the first and last equalities \eqref{p k and p +} imply the formula $p_+=R\rho\theta-\langle \alpha_kp_{*k}\rangle$ that, along with the preceding formula for $\rho(\ve-\ve_0)$, lead to the second
formula~\eqref{p and temp}.

\par In the simplest case of the perfect polytropic gases, i.e., $p_{*1}=p_{*2}=0$, we get
\begin{eqnarray*}
p_+=R\rho\theta=R_1\rho_1\theta+R_2\rho_2\theta,
\\[1mm]
\alpha_k=\frac{R_k\rho_k}{R\rho},\ \
\alpha_k p_+=R_k\rho_k\theta,\ \ k=1,2.
\end{eqnarray*}
Consequently, the above heterogeneous mixture model and the homogeneous one, with the different pressures $p_1=R_1\rho_1\theta$ and $p_2=R_2\rho_2\theta$ of the components occupying the same volume, become equivalent. A similar observation was given in \cite{CBS17CF}.
This explains why computations in \cite{ZFL22,ZL23,ESh22} using the homogeneous model led to the same results as in the papers based on the heterogeneous models.
\par Differentiating in the total momentum balance PDE \eqref{moment eq NS} and using the mass balance PDE \eqref{sum mass homog}, one derives the velocity balance PDE
\begin{eqnarray}
\partial_t\*u+(\*u\cdot\nabla)\*u+\rho^{-1}\nabla p=\rho^{-1}\dv\Pi^{NS}+\*f.
\label{velo homog}
\end{eqnarray}

\par For differentiable solutions, the systems of PDEs \eqref{mass eq alpha NS}--\eqref{energy eq NS} and
\eqref{mass eq alpha NS}, \eqref{velo homog} and \eqref{eq for p} are equivalent.
Recall that this allows one to give an easier analysis of the hyperbolicity properties in the case $\mu=\lambda=\vk=0$.
For simplicity, in the 1D case and for $\*f=Q=0$, the latter system can be written in the canonical matrix form
\begin{eqnarray*}
 \partial_t\*z+A\partial_1\*z=0,
 \\[1mm]
 \*z=(\rho_1,\rho_2,u,p)^T,\ \
 A:=\left(\begin{array}{cccc}
 u & 0 & \rho_1 & 0 \\
 0 & u & \rho_2 & 0 \\
 0 & 0 & u & \rho^{-1} \\
 0 & 0 & \rho c_s^2 & u \\
  \end{array}\right),
\end{eqnarray*}
with $u=u_1$.
One can easily calculate $\det(A-\lambda I)=(\lambda-u)^2[(\lambda-u)^2-c_s^2]$, and thus $A$ has the real eigenvalues $\lambda_{1,2}=u$ and $\lambda_{3,4}=u\pm c_s$.

\section{Regularized systems of PDEs for the dynamics of quasi-homogeneous mixtures of stiffened gases}

\par In this Section, we accomplish the formal regularization procedure first suggested in \cite{Z12MM2} for the single-component gas;
this procedure was shown to allow one to get a simple derivation of the quasi-gasdynamic (QGD) regularization described in \cite{E07}.

\par The procedure has recently been
used for the Euler-type system of PDEs for multicomponent one-velocity and one-temperature homogeneous mixture gas dynamics in Appendix A in \cite{ZL23}.
In the balance PDEs for the mass of components \eqref{mass eq alpha NS}, the total momentum \eqref{moment eq NS} and the total energy of the mixture \eqref{energy eq NS}, we accomplish respectively the following changes
\begin{eqnarray*}
 \rho_k\*u\ \rightarrow\ \rho_k\*u+\tau\partial_t(\rho_k\*u),
\label{replacement 1}
\\[1mm]
 \dv(\rho\*u\otimes\*u)+\nabla p-\rho\*f
 \rightarrow\
\dv(\rho\*u\otimes\*u+\tau\partial_t(\rho\*u\otimes\*u))
\\[1mm]
+\nabla (p+\tau\partial_t p)-(\rho+\tau\partial_t\rho)\*f
\label{replacement 2}
\end{eqnarray*}
and
\begin{eqnarray*}
(E+p)\*u\
 \rightarrow\ (E+p)\*u+\tau\partial_t\big((E+p)\*u\big),\\[1mm]
 \rho\*u\cdot\*f\ \rightarrow\ \big(\rho\*u+\tau\partial_t(\rho\*u)\big)\cdot\*f,
\label{replacement 3}
\end{eqnarray*}
where $E=(1/2)\rho|\*u|^2+\rho\ve$ is the total mixture energy and $\tau>0$ is a regularization parameter which can depend on all the sought functions.

\par These changes lead from the original Navier-Stokes-Fou\-ri\-er-type system
\eqref{mass eq alpha NS}-\eqref{energy eq NS}
to its following regularized QGD version
\begin{eqnarray}
\partial_t\rho_k+\dv\big(\rho_k(\*u-\*w_k)\big)=0,\ \ k=1,2,
\label{mass eq alpha QGD}\\[1mm]
\partial_t(\rho\*u)+\dv(\rho(\*u-\*w)\otimes\*u)+\nabla p
\nonumber\\[1mm]
=\dv(\Pi^{NS}+\Pi^\tau)+\big(\rho-\tau\dv(\rho\*u)\big)\*f,
\label{moment eq QGD}\\[1mm]
\partial_tE+\dv\big((E+p)(\*u-\*w)\big)
\nonumber\\[1mm]
=\dv(-\*q^F-\*q^\tau+(\Pi^{NS}+\Pi^\tau)\*u)+\rho(\*u-\*w)\cdot\*f+Q,
\label{energy eq QGD}
\end{eqnarray}
where the unknown functions are the same.
This system involves the regularizing velocities
\begin{eqnarray}
 \*w_k:=
 \frac{\tau}{\rho_k}\dv(\rho_k\*u)\*u+\widehat{\*w},\,
\widehat{\*w}
=\tau\Big((\*u\cdot\nabla)\*u+\frac{1}{\rho}\nabla p-\*f\Big),
\label{wk what}\\[1mm]
 \*w:=\Big\langle\frac{\rho_k}{\rho}\*w_k\Big\rangle=\frac{\tau}{\rho}\dv(\rho\*u\otimes \*u+\nabla p-\rho\*f)
\nonumber\\[1mm]
 =\frac{\tau}{\rho}\dv(\rho\*u)\*u+\widehat{\*w},
\label{w}
\end{eqnarray}
with $k=1,2$,
the regularizing viscous stress and heat flux
\begin{eqnarray}
\Pi^\tau:=\rho\*u\otimes\widehat{\*w}+\tau\Big(\*u\cdot\nabla p+\rho c_s^2\dv\*u-\frac{c_s^2}{\gamma c_V\theta}Q\Big)\mathbb{I},
\label{Pi tau}
\\
-\*q^\tau:=\tau\Big(\*u\cdot\Big(\rho\nabla\ve-\frac{p}{\rho}\nabla\rho\Big)-Q\Big)\*u.
\label{q tau}
\end{eqnarray}
In the last expression, using formula \eqref{sg EOS cons2b}, we can rewrite
\begin{eqnarray}
 \rho\nabla\ve-\frac{p}{\rho}\nabla\rho=\nabla(\rho\ve)-\frac{\rho\ve+p}{\rho}\nabla\rho\nonumber\\[1mm]
 =\nabla(\rho\ve)-(\gamma c_V\theta+\ve_0)\nabla\rho.
\label{term of q tau}
\end{eqnarray}
Actually, the derivation repeats the argument from Appendix A in \cite{ZL23}, its only difference being that another PDE for the pressure \eqref{eq for p} is used (for $\mu=\lambda=\vk=0$) and, in the expression for $\*q^\tau$, the fraction $p/\rho$ remains in its general form.
Note that the general form of Eqs. \eqref{mass eq alpha QGD}-\eqref{w} is the same as in \cite{ESh22,ZL23} but formulas \eqref{Pi tau}-\eqref{q tau} are different.
We emphasize that the form of the multiplier in front of $\dv\mathbf u$ in \eqref{Pi tau} is the same as in the case of single-component real gases in \cite{ZG11}.

\par For the regularized QGD system of PDEs, the additional balance PDEs for the mass, kinetic and internal energies of the mixture hold
\begin{eqnarray}
 \partial_t\rho + \dv(\rho(\*u-\*w))=0,
\label{sum mass regul}\\[1mm]
\tfrac12\partial_t(\rho|\*u|^2)+\tfrac12\dv\big(\rho|\*u|^2(\*u-\*w)\big)+\*u\cdot\nabla p
\nonumber\\[1mm]
=(\dv(\Pi^{NS}+\Pi^\tau))\cdot\*u+\big(\rho-\tau\dv(\rho\*u)\big)\*f\cdot\*u,
\label{kin en regul}\\[1mm]
 \partial_t(\rho\ve)+\dv(\rho\ve(\*u-\*w))+p\dv\*u
\nonumber\\[1mm]
 =\dv(-\*q^F-\*q^\tau+p\*w)+(\Pi^{NS}+\Pi^\tau):\nabla\*u
\nonumber\\[1mm]
 -\rho\*f\cdot\widehat{\*w}+Q
\label{int en regul}
\end{eqnarray}
which are derived similarly to the corresponding original PDEs \eqref{sum mass homog}-\eqref{int en homog}.

\par We also consider the
simplified regularized quasi-hyd\-ro\-dynamic (QHD) system of PDEs
\begin{eqnarray}
\partial_t\rho_k+\dv\big(\rho_k(\*u-\widehat{\*w})\big)=0,\ \ k=1,2,
\label{mass eq alpha QHD}\\[1mm]
\partial_t(\rho\*u)+\dv(\rho(\*u-\widehat{\*w})\otimes\*u)
+\nabla p
\nonumber\\[1mm]
=\dv(\Pi^{NS}+\widehat{\Pi}^\tau)+\rho\*f,
\label{moment eq QHD}\\[1mm]
\partial_tE+\dv\big((E+p)(\*u-\widehat{\*w})\big)
\nonumber\\[1mm]
=\dv\big(-\*q^F+(\Pi^{NS}+\widehat{\Pi}^\tau)\*u\big)+\rho(\*u-\widehat{\*w})\cdot\*f+Q.
\label{energy eq QHD}
\end{eqnarray}
Here, the regularizing velocity $\*w$ and viscous stress $\Pi^\tau$ are simplified as $\widehat{\*w}$ and
$\widehat{\Pi}^\tau:=\rho\*u\otimes\widehat{\*w},$
and the regularizing terms $\tau\dv(\rho\*u)$ and $-\*q^\tau$ are omitted. Recall that, in general, the QHD regularization shows marks of success in the cases where the Mach numbers are not high.

\par For the regularized QHD system of PDEs, the additional balance PDEs for the mass, kinetic and internal energies of the mixture hold
\begin{eqnarray}
 \partial_t\rho + \dv(\rho(\*u-\widehat{\*w}))=0,
\nonumber\\[1mm]
\tfrac12\partial_t(\rho|\*u|^2)+\tfrac12\dv\big(\rho|\*u|^2(\*u-\widehat{\*w})\big)+\*u\cdot\nabla p
\nonumber\\[1mm]
 =(\dv(\Pi^{NS}+\widehat{\Pi}^\tau))\cdot\*u+\rho\*f\cdot\*u,
\nonumber\\[1mm]
 \partial_t(\rho\ve)+\dv(\rho\ve(\*u-\widehat{\*w}))+p\dv\*u
\nonumber\\[1mm]
 =\dv(-\*q^F+p\widehat{\*w})+(\Pi^{NS}+\widehat{\Pi}^\tau):\nabla\*u-\rho\*f\cdot\widehat{\*w}+Q
\label{int en QHD}
\end{eqnarray}
which are simplified versions of the above corresponding PDEs \eqref{sum mass regul}-\eqref{int en regul}.

\par The $k$-th component specific entropy $s_k(r_k,\ve_k)$ is defined by the thermodynamic equations
\begin{eqnarray}
 \partial_{r_k}s_k=-\frac{p_k}{r_k^2\theta},\ \
 \partial_{\ve_k}s_k=\frac{1}{\theta},\ \ k=1,2;
\label{entropy def}
\end{eqnarray}
an explicit expression for $s_k$ is available but we do not need it.
Then the mixture specific entropy $s$ is given by the relation $\rho s=\langle\rho_k s_k\rangle$.
\begin{proposition}
For the regularized QHD system of PDEs, the balance PDE for the mixture entropy with the non-negative entropy production
\begin{eqnarray}
 \partial_t(\rho s)+ \dv\Big(\rho s(\*u-\widehat{\*w})+\frac{1}{\theta}\*q^F\Big)
\nonumber\\[1mm]
 =\frac{1}{\theta}\Big\{\frac{\mu}{2}|\nabla\*u+(\nabla\*u)^T\big|_F^2+\Big(\lambda-\frac23\mu\Big)(\dv\*u)^2
\nonumber\\[1mm]
 +\frac{\vk}{\theta}|\nabla\theta|^2]
 +\frac{\rho}{\tau}|\widehat{\*w}|^2+Q\Big\}\geq 0
\label{entropy eq}
\end{eqnarray}
is valid, where $|\cdot|_F$ is the Frobenius norm.
\end{proposition}
\begin{proof}
We use the balance PDE \eqref{mass eq alpha QHD} for the density of the $k$th component, the thermodynamic equations \eqref{entropy def}, the formula $\rho_k=\alpha_kr_k$ and then the balance PDE \eqref{mass eq alpha QHD} once more and get
\begin{eqnarray}
 \partial_t(\rho_ks_k)+ \dv(\rho_ks_k(\*u-\widehat{\*w}))=\rho_k\partial_ts_k+\rho_k(\*u-\widehat{\*w})\cdot\nabla s_k
\nonumber\\[1mm]
 =\rho_k\Big[-\frac{p_k}{r_k^2\theta}(\partial_tr_k+(\*u-\widehat{\*w})\cdot\nabla r_k)
\nonumber\\[1mm]
 +\frac{1}{\theta}(\partial_t\ve_k+(\*u-\widehat{\*w})\cdot\nabla \ve_k)
 \Big]
\nonumber\\[1mm]
=\frac{1}{\theta}\Big[-\frac{p_k}{r_k}\alpha_k(\partial_tr_k+(\*u-\widehat{\*w})\cdot\nabla r_k)
\nonumber\\[1mm]
+\partial_t(\rho_k\ve_k)+ \dv(\rho_k\ve_k(\*u-\widehat{\*w}))
\Big].
\label{entropy k eq}
\end{eqnarray}
Next, the same balance PDE \eqref{mass eq alpha QHD} for the density of the $k$th component with $\rho_k=\alpha_kr_k$ implies
\[
 \alpha_k(\partial_tr_k+(\*u-\widehat{\*w})\cdot\nabla r_k)
 =-\big(\partial_t\alpha_k+\dv(\alpha_k(\*u-\widehat{\*w})\big)r_k.
\]

\par We apply $\langle\cdot\rangle$ to Eq. \eqref{entropy k eq}, use the last formula and the equations $p_k=p$ and $\langle\alpha_k\rangle=1$ and find
\begin{eqnarray*}
 \partial_t\langle\rho_k s_k\rangle+ \dv(\langle\rho_k s_k\rangle(\*u-\widehat{\*w}))
\nonumber\\[1mm]
 =\frac{1}{\theta}\big[p\big(\partial_t\langle\alpha_k\rangle+\dv(\langle\alpha_k\rangle(\*u-\widehat{\*w})\big)
\\[1mm]
 +\partial_t\langle\rho_k\ve_k\rangle+ \dv(\langle\rho_k\ve_k\rangle(\*u-\widehat{\*w}))\big]
\nonumber\\[1mm]
 =\frac{1}{\theta}\big[p\dv(\*u-\widehat{\*w})
 +\partial_t(\rho\ve)+\dv(\rho\ve(\*u-\widehat{\*w}))\big].
\label{entropy eq 1}
\end{eqnarray*}

\par Finally, due to the balance PDE \eqref{int en QHD} for the 
internal energy of the mixture, we obtain
\begin{eqnarray*}
 \partial_t(\rho s )+ \dv(\rho s (\*u-\widehat{\*w}))
=\frac{1}{\theta}[\dv(-\*q^F)+\nabla p\cdot\widehat{\*w}
\\[1mm]
+(\Pi^{NS}+\rho\*u\otimes\widehat{\*w}):\nabla\*u-\rho\*f\cdot\widehat{\*w}+Q\big]
\nonumber\\
=\dv\Big(-\frac{1}{\theta}\*q^F\Big)+\frac{\vk}{\theta^2}|\nabla\theta|^2
+\frac{1}{\theta}\Big(\Pi^{NS}:\nabla\*u+\frac{\rho}{\tau}|\widehat{\*w}|^2+Q\Big),
\label{entropy eq 2}
\end{eqnarray*}
that implies the balance PDE for the mixture entropy \eqref{entropy eq}.
\end{proof}

Note that the specific stiffened gas equations of state were not used in the last Proposition.

\section{Finite-difference schemes for the 1D regularized systems of PDes}
\par 
Further, we consider the 1D case with $\Omega=(-X,X)$ and define the main and auxiliary uniform meshes
\begin{eqnarray*}
\bar{\omega}_h=\{x_i=-X+ih;\,0\leq i\leq N\},
\\[1mm]
\omega_h^*=\{x_{i+1/2}=-X+(i+0.5)h;\,0\leq i\leq N-1\},
\end{eqnarray*}
on $[-X,X]$, with the step $h=2X/N$.
We also define the nonuniform mesh $\bar{\omega}^{\Delta t}=\{t_0=0<t_1<\ldots<t_{\overline{m}}=t_{fin}\}$ in time, with the steps $\Delta t_{m}=t_{m+1}-t_m$.
Let $\omega_h=\bar{\omega}_h\backslash\{-X,X\}$ and $\check{\omega}^{\Delta t}=\bar{\omega}^{\Delta t}\backslash \{t_{fin}\}$.
\par Denote by $H(\omega)$ the space of functions given on a mesh $\omega$.
For $v\in H(\bar{\omega}_h)$, $w\in H(\omega_h^*)$ and $y\in H(\bar{\omega}^\tau)$,
we introduce the averages and difference quotients
\begin{eqnarray*}
 [v]_{i+1/2}=0.5(v_i+v_{i+1}),\ \
 v_{i+1/2}=\frac{v_{i+1}-v_i}{h},
\\[1mm]
 [w]^*_i=0.5(w_{i-1/2}+w_{i+1/2}),\ \
  \delta^*w_i=\frac{w_{i+1/2}-w_{i-1/2}}{h},
\\[1mm]
 \delta_ty^m=\frac{y^{m+1}-y^m}{\Delta t_{m}},
\end{eqnarray*}
where $v_i=v(x_i)$, $w_{i+1/2}=w(x_{i+1/2})$ and $y^m=y(t_m)$.
Let also $v_{-,i+1/2}=v_i$, $v_{+,i+1/2}=v_{i+1}$ and $\hat{y}^m=y^{m+1}$.

\par For simplicity, suppose that
the body force is absent: $\*f=0$ (the general case can be covered as well, see \cite{Z12CMMP,ZL23}).
We consider the regularized QGD balance PDEs \eqref{mass eq alpha QGD}--\eqref{energy eq QGD} in the 1D case
and construct the following explicit two-level in time and symmetric three-point in space discrete balance equations for the mass of the components and the momentum and total energy of the gas mixture
\begin{eqnarray}
 \delta_t\rho_k+\delta^*\big([\rho_k]([u]-w_k)\big)=0,\ \ k=1,2,
\label{mass dis}\\[1mm]
\delta_t(\rho u)+\delta^*\big([\rho]([u]-w)[u]+[p]\big)=\delta^*\Pi,
\label{imp dis}\\[1mm]
\delta_t\big(\tfrac12\rho u^2+\rho\ve\big)
\nonumber\\[1mm]
+\delta^*\big\{\big(\tfrac12\rho u_-u_+ +[\rho\ve]+[p]\big)([u]-w)
\nonumber\\[1mm]
-\tfrac14h^2(\delta p)\delta u\big\}
=\delta^*(-q+\Pi[u])+[Q]^*
\label{energy dis}
\end{eqnarray}
on $\omega_h\times\check{\omega}^{\Delta t}$.
Here the main sought functions $\rho_1>0,\rho_2>0$, $u$ and $\ve$ (in fact, $\rho\ve$), along with the functions $p$ and $\theta$, are defined on the main mesh $\bar{\omega}_h\times\bar{\omega}^{\Delta t}$.
Also
$p$ and $\theta$ are given respectively by the first formula \eqref{p and temp} and formula \eqref{2nd form for temp},
with $d=b^2+4c$ (or see formula \eqref{second form for d} for $d$) and their coefficients
defined by \eqref{rho sigmak}, \eqref{coeff b} and \eqref{coeff c}.
In computations below, we also use formula \eqref{form for alpha k} for $\alpha_k$.

\par In Eq. \eqref{energy dis}, the nonstandard term $u_-u_+$ (close to the geometric mean for $u^2$) instead of $[u^2]$ or $[u]^2$ and the additional term
$-(1/4)h^2(\delta p)\delta u$ allows us to ensure a more natural form of the important discrete balance equation for the mixture internal energy like in \cite{Z12CMMP} without the spatial mesh imbalances, see
Proposition~\ref{prop: add discr balance eqs} below.

\par We discretize the regularizing velocities \eqref{wk what}-\eqref{w} in the form
\begin{eqnarray}
w_k=\frac{[\tau]}{[\rho_k]}\,[u]\delta(\rho_k u)+\widehat{w},\ \
 \widehat{w}=\frac{[\tau]}{[\rho]}\,([\rho][u]\delta u+\delta p),
\label{w k discr}\\[1mm]
 w=\Big\langle\frac{[\rho_k]}{[\rho]}w_k\Big\rangle
 =\frac{[\tau]}{[\rho]}\,[u]\delta(\rho u)+\widehat{w}
\label{w discr}
\end{eqnarray}
with $k=1,2$ and the viscous stress and heat flux, see \eqref{Pi NS q F} and \eqref{Pi tau}-\eqref{term of q tau}, as follows
\begin{eqnarray}
 \Pi=\nu\delta u
 +[u][\rho]\widehat{w}
\nonumber\\[1mm]
 +[\tau]\Big([u]\delta p+[\rho c_s^2]\delta u-\frac{[c_s^2]}{[\gamma c_V][\theta]}Q\Big),
\label{pi hom dis}
\end{eqnarray}
\begin{eqnarray}
-q=\vk\delta\theta
\nonumber\\[1mm]
 +[\tau]\Big\{\Big(\delta(\rho\ve)-\frac{[\rho\ve]+[p]}{[\rho]}\delta\rho\Big)[u]^2-Q[u]\Big\}.
\label{q tau hom dis}
\end{eqnarray}
Here the squared speed of sound $c_s^2$ is given by the second formula
in \eqref{for speed of sound}, and $c_V$ and $\gamma$ are introduced in \eqref{coeff of mixture}.
The functions $w_k$, $\widehat{w}$, $w$,
$\Pi$, $\nu=(4/3)\mu+\lambda$, $q$, $\vk$ and $Q$ are defined on the auxiliary mesh $\omega_h^*\times\bar{\omega}^{\Delta t}$, but $\tau$ is defined on $\bar{\omega}_h\times\bar{\omega}^{\Delta t}$.
We take $\tau$, $\nu$ and $\vk$ in the form
\begin{eqnarray}
 \tau=\frac{ah}{c_s+i_\tau|u|},\ \
 \nu= a_S[\tau][p],\ \
 \vk= a_{Pr}[\tau][c_p][p]
\label{tau mu lam kappa}
\end{eqnarray}
that is formally analogous to the single-component gas case \cite{E07}.
So $\tau$ is $h$-dependent, $\nu$ and $\vk$ are artificial viscosity coefficients,
with the parameter $a>0$, the Schmidt and inverse Prandtl numbers for the mixture $a_S\geq 0$ and $a_{Pr}>0$
used as adjusting numerical parameters, and $i_\tau=0$ or 1.
In computations below, we only need $i_\tau=0$.
\par The initial data $(\rho_1,\rho_2,u,\rho\ve)=(\rho_1^0,\rho_2^0,u^0,(\rho\ve)^0)$ (or equi\-va\-lent ones) are given on $\bar{\omega}_h$.
One can find on $\omega_h$ sequentially $\hat{\rho}_1$ and $\hat{\rho}_2$ from Eq. \eqref{mass dis}, next $\hat{\rho}\hat{u}$ and then $\hat{u}$ from Eq. \eqref{imp dis} and
finally $(1/2)\hat{\rho}\hat{u}^2+\widehat{\rho\ve}$ and then $\widehat{\rho\ve}$ from Eq. \eqref{energy dis}.
In computations below, we also
put the boundary values $\vp_0=\vp_1$ and $\vp_N=\vp_{N-1}$ for $\vp=\hat{\rho}_1,\hat{\rho}_2,\hat{u}$ and $\widehat{\rho\ve}$.

\par The above spatial discretization is notably simpler than the entropy correct one constructed in Section~4 in \cite{ZL23} in the case of the homogeneous mixture of the perfect polytropic gases.
The discretization of the total energy balance PDE in \cite{ZL23} has been based on the original regularized QGD multi-velocity and multi-temperature model, but it is not yet available for the heterogeneous mixtures.
Also, here we use the simplest averages of $\rho_k$ and $\rho\ve$ in all the terms.
\begin{proposition}
\label{prop: add discr balance eqs}
The following discrete balance equations for the mass, kinetic and internal energies of the mixture
\begin{eqnarray}
 \delta_t\rho+\delta^*j=0,\ \ j:=[\rho]([u]-w),
\label{mass mix dis}\\[1mm]
 \tfrac12\delta_t(\rho u^2)-\tfrac{\Delta t}{2}\hat{\rho}(\delta_tu)^2
 +\tfrac12\delta^*(ju_-u_+)+(\delta^*[p])u
\nonumber\\[1mm]
 =(\delta^*\Pi)u,
\label{kin energy mix dis}\\[1mm]
  \delta_t(\rho\ve)+\tfrac{\Delta t}{2}\hat{\rho}(\delta_tu)^2
 +\delta^*(j[\ve])
\nonumber\\[1mm]
 =-\delta^*q+[\Pi\delta u]^*-p\delta^*([u]-w)+[w\delta p]^*+[Q]^*
\label{internal energy mix dis}
\end{eqnarray}
are valid on $\omega_h\times\check{\omega}^{\Delta t}$, cf. the corresponding balance PDEs \eqref{sum mass regul}, \eqref{kin en regul} and \eqref{int en regul}.
\end{proposition}
\begin{proof}
Applying $\langle\cdot\rangle$ to the discrete balance Eq. \eqref{mass dis} and using formula \eqref{w discr}, we derive the discrete balance equation for the mixture mass \eqref{mass mix dis}.

\par We also multiply the discrete balance Eq. \eqref{mass dis} by $u$ and apply the known formula
\[
 \delta_t(\rho u)
 =\tfrac12\delta_t(\rho u^2)+\tfrac12(\delta_t\rho)u^2-\tfrac12\Delta t\hat{\rho}(\delta_tu)^2.
\]
Using also Eq. \eqref{mass mix dis} and transformations
\begin{eqnarray*}
 \tfrac12(\delta_t\rho)u^2=-\tfrac12(\delta^*j)u^2=-\tfrac12\big(\delta^*(j[u^2])-[j\delta(u^2)]^*\big)
\\[1mm]
 =-\delta^*(j[u]^2)+\tfrac12\delta^*(ju_-u_+)+[j[u]\delta u],
\\[1mm]
 \delta^*(j[u])u=\delta^*(j[u]^2)-[j[u]\delta u]^*,
\end{eqnarray*}
see Section 2.2 in \cite{Z12CMMP}, we get the discrete balance equation for the 
kinetic energy of the mixture
\eqref{kin energy mix dis}.

\par Subtracting it from the discrete balance Eq. \eqref{energy dis} for the total energy and using formulas
\begin{eqnarray*}
 \delta^*\big([p][u]-\tfrac14h^2(\delta p)\delta u\big)=(\delta^*[p])u+p\delta^*[u],
\\[1mm]
 \delta^*(w[p])=(\delta^*w)p+[w\delta p]^*,\\[1mm]
 \delta^*(\Pi[u])=(\delta^*\Pi)u+[\Pi\delta u]^*.
\end{eqnarray*}
see Section 2.2 in \cite{Z12CMMP}, we obtain the last discrete balance Eq.
\eqref{internal energy mix dis}.
\end{proof}

\par Below the time steps are chosen automatically according to the formulas
\begin{eqnarray}
\Delta t_{m}=\frac{\beta h}{\max_i(c_{si}^m+|u_i^m|)},\ 0\leq m<\overline{m}-1,
\label{dtm}\\[1mm]
\Delta t_{\overline{m}-1}=t_{fin}-t_{\overline{m}-1}\leq\frac{\beta h}{\max_i(c_{si}^{\overline{m}-1}+|u_i^{\overline{m}-1}|)},
\label{dtmneg}
\end{eqnarray}
where $\beta>0$ is the Courant-type parameter.
Note that the conditions for linearized $L^2$-dissipativity of the constructed scheme in the case of the single-component perfect polytropic gas follow from \cite{ZL18}.

\par For the simplified QHD regularization, the above constructed scheme is reduced as
\begin{eqnarray*}
 \delta_t\rho_k+\delta^*\big([\rho_k]([u]-\widehat{w})\big)=0,\ \ k=1,2,
\label{mass dis qhd}\\[1mm]
\delta_t(\rho u)+\delta^*\big([\rho]([u]-\widehat{w})[u]+[p]\big)=\delta^*\widehat{\Pi},
\label{imp dis qhd}\\[1mm]
\delta_t\big(\tfrac12\rho u^2+\rho\ve\big)
\nonumber\\[1mm]
+\delta^*\big\{\big(\tfrac12\rho u_-u_+ +[\rho\ve]+[p]\big)([u]-\widehat{w})-\tfrac14h^2(\delta p)\delta u\big\}
\nonumber\\[1mm]
=\delta^*(\vk\delta\theta+\widehat{\Pi}[u])+[Q]^*
\label{energy dis qhd}
\end{eqnarray*}
on $\omega_h\times\check{\omega}^{\Delta t}$, with $\widehat{\Pi}:=[u][\rho]\widehat{w}$ and $\widehat{w}$ defined in \eqref{w k discr}.
Below we apply it as well.

\section{Numerical experiments}

The aim of this Section is to verify practically the above constructed QGD and QHD regularizations. As usual, we
first accomplish it in the 1D case. We present results of seven numerical experiments for tests with shock waves.
The tests are taken from \cite{KLC14,LF11,CBS17IJNMF,YC13,LA12,ABR20}. The physical interpretation of the tests presented is discussed in the referenced sources, and we do not dwell much on it. The stiffened gas parameters for all the tests are collected in Table \ref{tab:sgparams}.
\begin{table}
\caption{\label{tab:sgparams}Stiffened gas parameters in seven tests}
\begin{ruledtabular}
\begin{tabular}{ccccc}
{Substance}&$\gamma$
&$c_v$, J/(kg K)&$p_\infty$, Pa&  $\ve_0$, J/kg \\[1mm]
\hline
\multicolumn{5}{l}{{A. Air-to-water shock tube problem}\cite{KLC14}}\\[1mm]
\hline
Air & 1.4 & 717.5&0&0\\[1mm]
Water & 2.8 & 1495&$8.5\cdot 10^8$&0\\[1mm]
\hline
\multicolumn{5}{l}{{B. Water-to-air shock tube problem}\cite{LF11}}\\[1mm]
\hline
Air & 1.4 & 720&0&0\\[1mm]
Water & 2.8 & 1495&$8.5\cdot 10^8$&0\\[1mm]
\hline
\multicolumn{5}{l}{{C. Shock tube test with a mixture containing mainly water vapor}\cite{CBS17IJNMF,BCPCA22}}\\[1mm]
\multicolumn{5}{l}{{D. Shock tube test with a vanishing liquid phase}\cite{CBS17IJNMF}}\\[1mm]
\multicolumn{5}{l}{{E. Shock tube test with a mixture containing mainly liquid water}\cite{CBS17IJNMF}}\\[1mm]
\hline
Water vapor & 1.43 & 1040&0&$\hphantom{-}2030\cdot 10^3$\\[1mm]
Water & 2.35 & 1816 &$10^9$&$-1167\cdot 10^3$\\[1mm]
\hline
\multicolumn{5}{l}{{F. Dodecane vapor-to-liquid shock tube}\cite{YC13}}\\[1mm]
\hline
Vapor & 1.025 & 1956 &0&$-237 \cdot 10^3$\\[1mm]
Liquid & 2.35 & 1077&$4\cdot 10^8$&$-755\cdot 10^3$\\[1mm]
\hline
\multicolumn{5}{l}{{G. Carbon dioxide depressurization}\cite{LA12,ABR20}}\\[1mm]
\hline
Vapor & 1.06 & 2410 &$8.86\cdot 10^5$&$-3.01 \cdot 10^5$\\[1mm]
Liquid & 1.23 & 2440&$1.32\cdot 10^8$&$-6.23\cdot 10^5$\\[1mm]
\end{tabular}
\end{ruledtabular}
\end{table}

\subsection{Air-to-water shock tube problem}

A 10 m long tube is separated into two halves, and the initial discontinuity is located in the middle of the tube. The left half is filled with air, and the right half is filled with water. The initial conditions are given by the formulas
\[
(p_0,u_0,\theta_0)=
\begin{cases}
(10^{9}\text{ Pa}, 0 \text{ m/s}, 308.15\text{ K}), & -5 \leq x < 0
\\[1mm]
(10^5\text{ Pa}, 0 \text{ m/s}, 308.15\text{ K}), & 0 < x \leq 5
\end{cases}.
\]

For numerical purposes, we use almost pure phases: $\alpha_1 = 1-10^{-5}$ in the left half and $\alpha_1=10^{-5}$ in the right half.
The plots that show the total density $\rho$ of the mixture, along with the mass $y_1$ and volume $\alpha_1$ fractions of the gas phase, pressure $p$, velocity $u$, and absolute temperature $\theta$ of the mixture, are depicted in Fig. \ref{fig:kitamura} for the QGD regularization and final time $t_{fin}=2$ ms (the same sought functions for other tests are presented in subsequent figures except for tests C to E where the more representative plots of $\alpha_2$ are given instead of $\alpha_1$). Recall that $N$ is the number of partition segments of $\overline{\Omega}$, the parameters  $a_{S}$, $a_{Pr}$ and $\beta$ are used in formulas \eqref{tau mu lam kappa} and \eqref{dtm}-\eqref{dtmneg}.
The most standard values of the Schmidt and inverse Prandtl numbers $a_{S}=1$ and $a_{Pr}=1$ are taken, and quality of the solution is quite good. The same values are taken below except where noted.

Note that we have a sort of a parasitic invariability segment in the velocity around the rarefaction wave in air. The plot of the average velocity in \cite{KLC14} also had a segment that slightly differed from the invariability segment in the solution, but that defect occurred around the shock in water.
It seems that the reason of this effect is that we use a one-velocity and one-temperature model instead of the two-velocity and two-temperature one from \cite{KLC14}.
Nonetheless, although the corresponding six-equation model is far more complicated compared to the four-equation one presented in this paper, the results of the numerical experiments are quite close. In this test, notice that $y_1$ and $\alpha_1$  practically coincide, and this  occasional circumstance
explains the success of computations in \cite{ESh22}.

Also, if we take $a_S=0$, the computation runs normally and quality of the solution is preserved. In the case of the QHD regularization, quality of the computed pressure and temperature is slightly worse. In general, the scope of applicability for the QHD regularization turns out to be narrower than for the QGD one, and the QGD regularization allows achieving better results.
\begin{figure*}[hbt!]
\center{\includegraphics[width=1\linewidth]{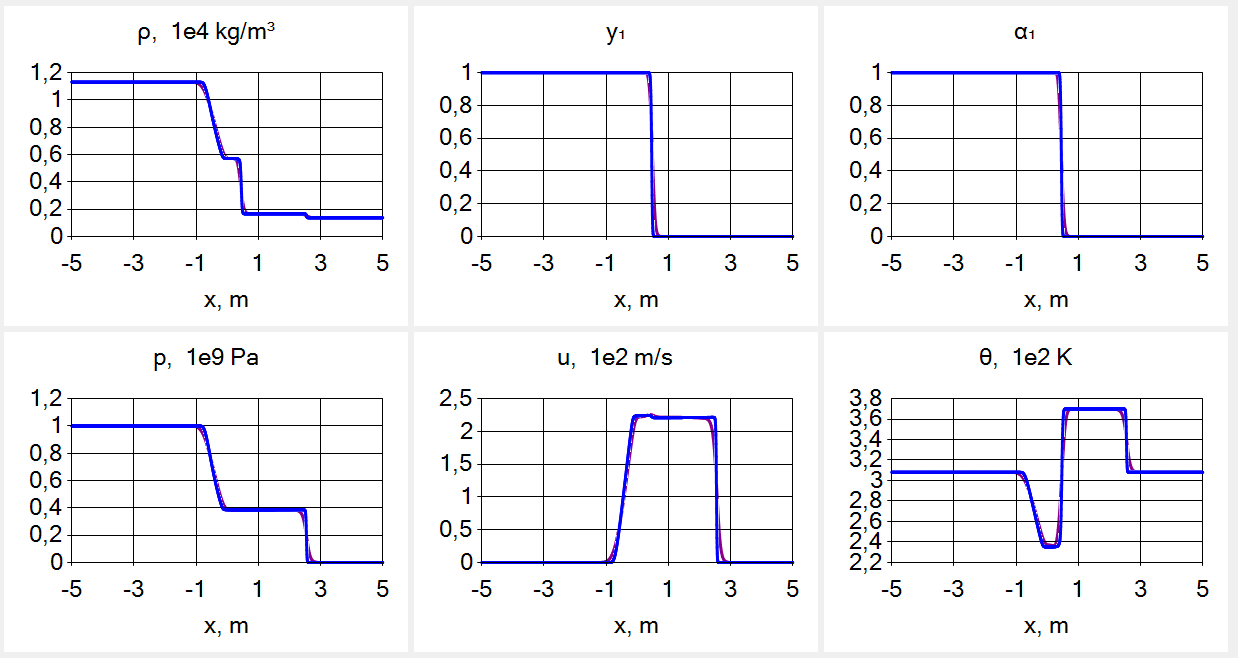}}
\caption{Numerical results for air-to-water shock tube problem
(test A for $N= 300$ (dark magenta), $2000$ (blue), $a=0.3$ and $\beta=0.2$ (the QGD regularization)}
\label{fig:kitamura}
\end{figure*}

\subsection{Water-to-air shock tube problem}

In this test, we again have a 10 m long tube separated into two halves, both of which contain a mixture of air and water but in different proportions.  The initial conditions are given by the formulas
\[
(p_0,u_0,\theta_0)=
\left\{
\begin{array}{rr}
(2\cdot10^{7}\text{ Pa}, 0 \text{ m/s}, 308.15\text{ K}), & -5 \leq x < 0\\[1mm]
\hphantom{2\cdot\,}(10^7\text{ Pa}, 0 \text{ m/s}, 308.15\text{ K}), & 0 < x \leq 5
\end{array}
\right.,
\]
 and we have $\alpha_1=0.25$ in the left half and $\alpha_1=0.75$ in the right half. The results are presented for $t_{fin}=6$ ms in Fig. \ref{fig:lifu} for the QGD regularization.

Notice that $a=2$ is taken, thus, $a>1$.
Although the computation does not fail if we take $0<a<1$ (that is the most often used interval), the quality of the solutions turns out to be much worse. The greater we take the value of $a$, the more smoothed the solution becomes.
For $a=0.5$ and $N=500$, we observe several oscillations in $\rho$, $y_1$ and $\alpha_1$ in a very close proximity to the left of the shock. Those oscillations diminish but do not vanish completely for $N=2500$ and look as ``fingers'' of a rather small height at the point of discontinuity.
For $a=1$, those negative effects have a lesser scale, but they are still observable; we omit the corresponding figures for brevity.

\par If we take $a_S=0$, the computation runs normally and quality of the solution is preserved.
In this test, we also see that $y_1$ and $\alpha_1$ are essentially different.
The QHD regularization fails to compute this test.
\begin{figure*}[hbt!]
\center{\includegraphics[width=1\linewidth]{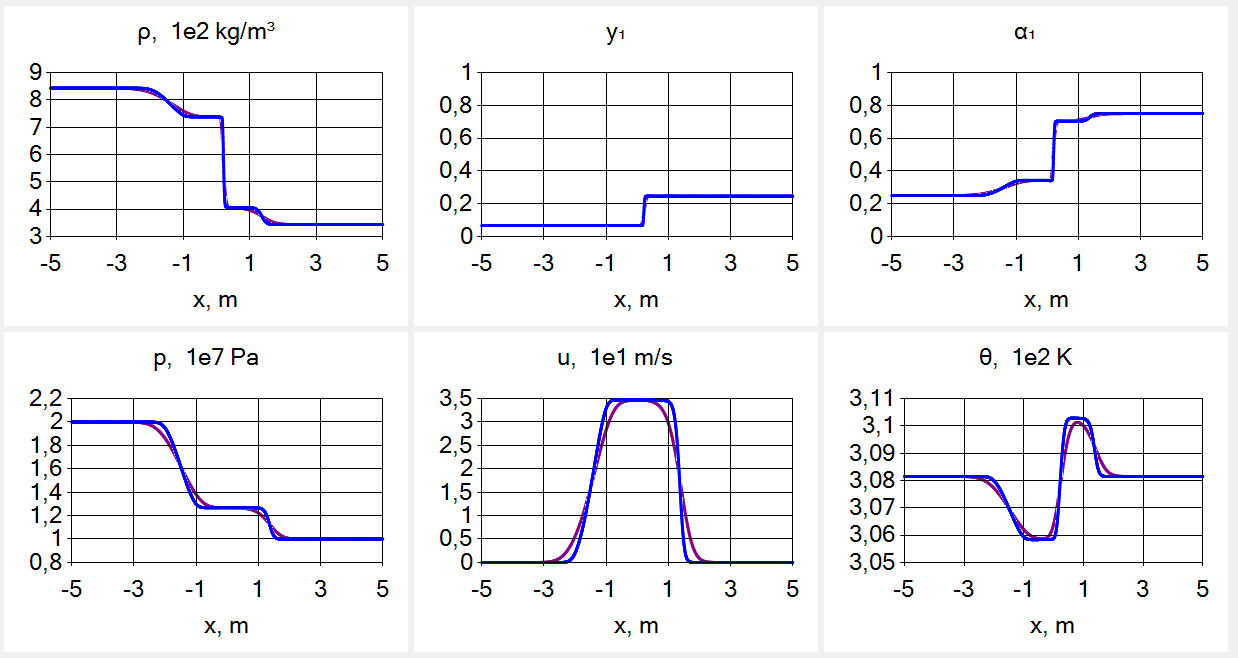}}
\caption{Numerical results for water-to-air shock tube (test B)
 for $N= 500$ (dark magenta), $2500$ (blue), $a=2$ and $\beta=0.1$ (the QGD regularization)}
\label{fig:lifu}
\end{figure*}

\par In this test, we also study an error of the constructed scheme.
Since the exact solution is unavailable, in a standard manner, we first compute the pseudo-exact solution for the fine mesh with $N=32000$.
Then we compute the numerical solutions for $N=250,500,1000,2000,4000$ and find the corresponding scaled (divided by $2X$) mesh $L^1$-norms $e_N(r)$ of the difference between the pseudo-exact solution and numerical one, for the functions $r=\rho$, $y_1$, $\alpha_1$, $p$, $u$ and $\theta$, together with the
corresponding practical error orders
$o(r)=o_N(r)=\log_2\big(e_{N/2}(r)/e_N(r)\big)$ for $N\geq 500$.
The results are put in Tables II and III.
The errors monotonically decrease and the orders slightly increase as $N$ grows.
Notice that
$0.476\leq o_{500}(r)\leq 0.585$,
$0.607\leq o_{1000}(r)\leq 0.679$,
$0.621\leq o_{2000}(r)\leq 0.702$ and
$0.720\leq o_{4000}(r)\leq 0.805$
for all $r$ that is rather normal since the exact solution is discontinuous.
Moreover, $o_N(p)$ and $o_N(u)$ are very close and are maximal for all $N$ except $N=1000$ whereas $o_N(y_1)$ is minimal for all $N$ except $N=1000$.
\begin{table}
\caption{\label{tab:errors 1}$L_1$-errors and error orders for $\rho$, $y_1$ and $\alpha_1$ (test B)}
\begin{ruledtabular}
\begin{tabular}{ccccccc}
$N$ & $e_N(\rho)/10^2$ & $o(\rho)$ & $e_N(y_1)$ & $o(y_1)$ & $e_N(\alpha_1)$ & $o(\alpha_1)$\\[1mm]
\hline
 250 & 7.5213E-02 & -- & 1.0362E-03 & -- & 6.6917E-03 & --\\[1mm]
 \hline
 500 & 5.1918E-02 & 0.535 & 7.4500E-04 & 0.476 & 4.5815E-03 & 0.547\\[1mm]
 \hline
1000 & 3.3670E-02 & 0.625 & 4.6544E-04 & 0.679 & 2.9417E-03 & 0.639\\[1mm]
 \hline
 2000 & 2.1165E-02 & 0.670 & 3.0272E-04 & 0.621 & 1.8497E-03 & 0.669\\[1mm]
 \hline
 4000 & 1.2372E-02 & 0.775 & 1.8382E-04 & 0.720 & 1.0827E-03 & 0.773\\[1mm]
 \end{tabular}
\end{ruledtabular}
\end{table}
\begin{table}
\caption{\label{tab:errors 2}$L_1$-errors and error orders for $p$, $u$ and $\theta$ (test B)}
\begin{ruledtabular}
\begin{tabular}{ccccccc}
 $N$ & $e_N(p)/10^7$ & $o(p)$ & $e_N(u)/10^1$ & $o(u)$ & $e_N(\theta)/10^2$ & $o(\theta)$\\[1mm]
\hline
 250 & 3.5121E-02	& -- & 2.5293E-01 & -- &	2.648273E-03 & --\\[1mm]
 \hline
 500 & 2.3413E-02	& 0.585	& 1.6859E-01&0.585&1.804264E-03&	0.554\\[1mm]
 \hline
1000 & 1.5060E-02 & 0.637 & 1.0832E-01 & 0.638 & 1.184704E-03	& 0.607\\[1mm]
 \hline
 2000 & 9.2575E-03 & 0.702 & 6.6605E-02 & 0.702 & 7.476341E-04 & 0.664\\[1mm]
 \hline
 4000 & 5.3013E-03 & 0.804 & 3.8112E-02 & 0.805 & 4.425483E-04	& 0.756\\[1mm]
\end{tabular}
\end{ruledtabular}
\end{table}

\subsection{Shock tube test with a mixture containing mainly water vapor}

We take a 1 m long tube filled with a mixture that contains mainly water vapor with the mass fraction $y_1=0.8$ in the entire tube. In this test, and tests D and E below as well, the volume fraction $\alpha_1$ is computed via formula \eqref{vol fr mass fr}. The initial conditions are as follows
\begin{gather*}
(p_0,u_0,\theta_0)=
\left\{
\begin{array}{rr}
(2\cdot10^{5}\text{ Pa}, 0 \text{ m/s}, 394.2489\text{ K}),  -5 \leq x < 0\hphantom{.}
\\[1mm]
(10^5\text{ Pa}, 0 \text{ m/s}, 372.8827\text{ K}),\hphantom{-}0 < x \leq 5.
\end{array}
\right.
\end{gather*}

The results are shown for $t_{fin}=0.8$ ms in Fig. \ref{fig:chiapolino61qhd} for the QHD regularization, and they are in perfect agreement with \cite{CBS17IJNMF,BCPCA22}.
Note that $y_1$ and $\alpha_1\approx 1$ are both almost constant but they are different; also $\alpha_2$ is rather small, but its behavior is nontrivial.
Hereafter, if the results for the QHD regularization are presented, then the corresponding results for the QGD regularization are always
of at least the same quality.

Also, in the QGD case, if we increase the value of $a$, we can also take $a_S=0$. Quality of the solution remains at the same level.
\begin{figure*}[hbt!]
\center{\includegraphics[width=1\linewidth]{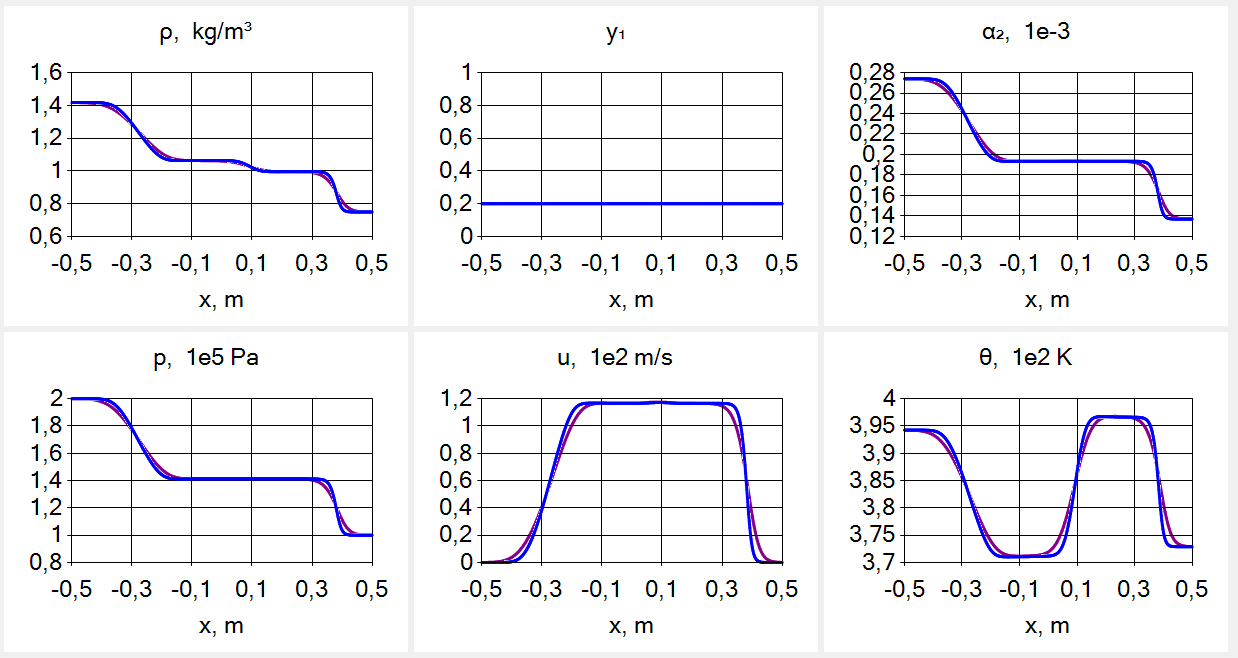}}
\caption{Numerical results for shock tube with a mixture containing mainly water vapor
(test C) for $N= 200$ (dark magenta), $500$ (blue), $a=0.8$ and $\beta=0.2$ (the QHD regularization)}
\label{fig:chiapolino61qhd}
\end{figure*}

\subsection{Shock tube test with a vanishing liquid phase}

In this test, a 1 m long tube is filled with a mixture with an almost vanishing liquid phase ($y_2=0.01$ in the entire tube). The initial conditions are given by the formulas
\[
(p_0,u_0,\theta_0)=
\left\{
\begin{array}{rr}
(2\cdot10^{5}\text{ Pa}, 0 \text{ m/s}, 395\text{ K}), & -5 \leq x < 0\\[1mm]
(10^5\text{ Pa}, 0 \text{ m/s}, 375\text{ K}), & 0 < x \leq 5
\end{array}
\right..
\]

The results are presented for $t_{fin}=0.5$ ms in Fig. \ref{fig:chiapolino62} for the QHD regularization, and they are in perfect agreement with \cite{CBS17IJNMF}.
Note that $y_1$ and $\alpha_1\approx 1$ are both almost constant and close to each other; also though $\alpha_2$ is even smaller than in test C, its behavior is still nontrivial.
In addition, in the QGD case, if we increase the value of $a$, we can also take $a_S=0$; quality of the solution remains at the same level.
\begin{figure*}[hbt!]
\center{\includegraphics[width=1\linewidth]{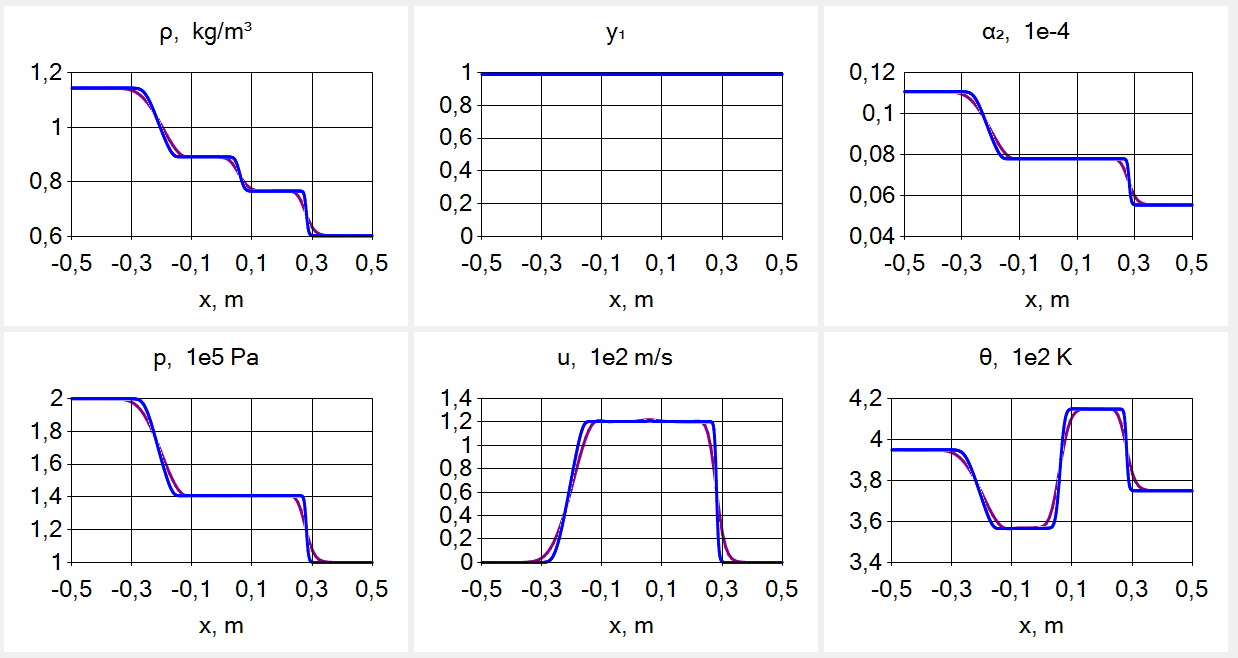}}
\caption{Numerical results for shock tube with a vanishing liquid phase  (test D)
 for $N=100$ (dark magenta), $500$ (blue), $a=0.2$ and $\beta=0.2$ (the QGD regularization)}
\label{fig:chiapolino62}
\end{figure*}

\subsection{Shock tube test with a mixture containing mainly liquid water}
We deal with a 1 m long tube filled with a mixture that contains mainly liquid water with $y_2=0.8$ in the entire tube.

The initial conditions are as follows
\[
(p_0,u_0,\theta_0)=
\left\{
\begin{array}{rr}
(2\cdot10^{5}\text{ Pa}, 0 \text{ m/s}, 395\text{ K}), & -5 \leq x < 0\\[1mm]
(10^5\text{ Pa}, 0 \text{ m/s}, 375\text{ K}), & 0 < x \leq 5
\end{array}
\right..
\]

The results are demonstrated for $t_{fin}=1.5$ ms in Fig. \ref{fig:chiapolino63qhd} for the QHD regularization, and they are in perfect agreement with \cite{CBS17IJNMF}.
The situation with $y_1$, $\alpha_1\approx 1$ and $\alpha_2$ is close to test C, though now $\alpha_2$ is larger.
Once again, in the QGD case, if we increase the value of $a$, we can also take $a_S=0$; quality of the solution remains at the same level.
\begin{figure*}[hbt!]
\center{\includegraphics[width=1\linewidth]{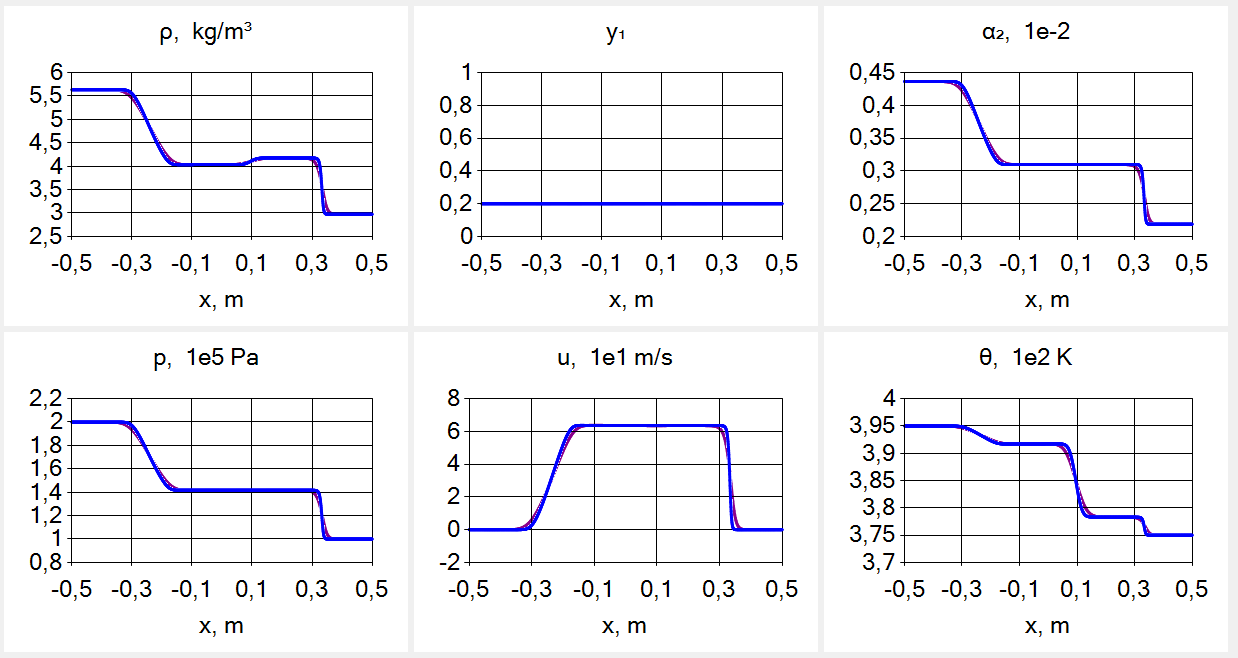}}
\caption{Numerical results for shock tube with a mixture containing mainly liquid water
(test E) for $N= 500$ (dark magenta), $1500$ (blue), $a=0.8$ and $\beta=0.3$ (the QHD regularization)}
\label{fig:chiapolino63qhd}
\end{figure*}

\subsection{Dodecane vapor-to-liquid shock tube}
We consider the dodecane vapor-liquid shock tube solved in \cite{YC13}: a 10 m long shock tube is filled with vapor dodecane under high pressure at the left, and with the liquid dodecane under atmospheric pressure at the right.
The initial discontinuity is set at a distance of 3 m from the left end, and the initial conditions are given by
\[
(p_0,u_0,\theta_0)=
\left\{
\begin{array}{rr}
(10^{10}\text{ Pa}, 0 \text{ m/s}, 308.15\text{ K}), & -5 \leq x < -2\\[1mm]
(10^5\text{ Pa}, 0 \text{ m/s}, 308.15\text{ K}), & -2 < x \leq 5
\end{array}
\right..
\]
The pure fractions are used, i.e., $\alpha_1=1$ in the left half and $\alpha_1=0$ in the right half. In this test, the parameters $a_{Pr}^{-1}=1$ and $0.2$ were taken. Quality of the results is better in the latter case. The results of those computations are presented for $t_{fin}=5$ ms in Fig. \ref{fig:yeomchang} for the QGD regularization.
The results obtained from the QHD regularization are slightly worse than in the QGD case.
Those results have small nonphysical gaps in the plots of temperature and gas volume fraction.
Nonetheless, the overall quality of the solution under this regularization remains at an acceptable level.
Here, $y_1$ and $\alpha_1$ are again close to each other like in test~A.

The six-equation system used in \cite{YC13}  is more complicated than the four-equation one presented in this paper. The numerical profiles of the depicted functions correspond well to those presented in \cite{YC13} regarding quality, although the shock wave in our model propagates faster.

Also, in the QGD case, we can take $a_S=0$ without increasing $a$ and with the same quality of the results.
\begin{figure*}[hbt!]
\center{\includegraphics[width=1\linewidth]{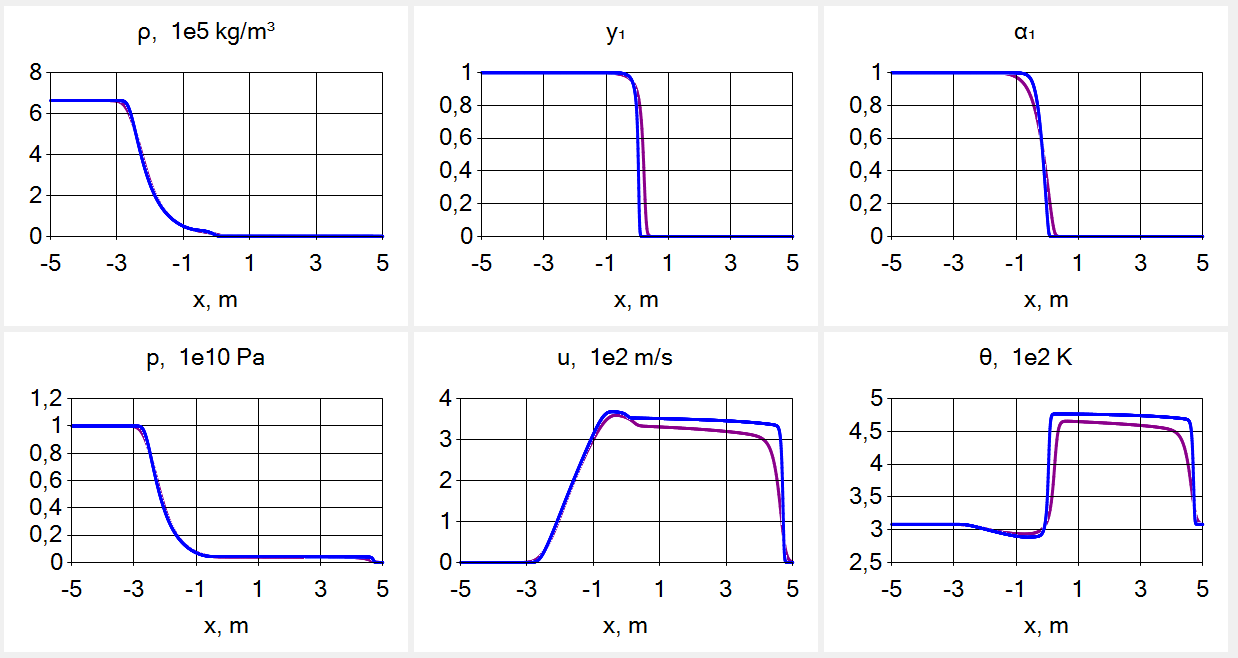}}
\caption{Numerical results for dodecane vapor-to-liquid shock tube (test F)
 for $N=500$ (dark magenta), $2000$ (blue), $a=0.9$ and $\beta=0.1$ (the QGD regularization)}
\label{fig:yeomchang}
\end{figure*}

\subsection{Carbon dioxide depressurization}
Here, we simulate depressurization of 
a 160-m long pipe that is filled with pure carbon dioxide, see \cite{LA12,ABR20}. The comparison of the results with the seven-equation model is given in \cite{LA12}. The pipe is filled with liquid carbon dioxide at the left, and with the gas carbon dioxide at the right.
The initial discontinuity is set at a distance of 50 m from the left, and the initial conditions are as follows:
\[
(p_0,u_0,\theta_0)=
\left\{
\begin{array}{rr}
(6\cdot10^{6}\text{ Pa}, 0 \text{ m/s}, 283.13\text{ K}),  -40 \leq x < 10\hphantom{.}\\[1mm]
(10^6\text{ Pa}, 0 \text{ m/s}, 283.13\text{ K}),   \hphantom{-}10 < x \leq 40.
\end{array}
\right.
\]

The results are presented for $t_{fin}=0.08$ s in Fig. \ref{fig:abgrall} for the QGD regularization. For numerical purposes, we use almost pure phases: $\alpha_1 = 10^{-6}$ in the left half and $\alpha_1=1-10^{-6}$ in the right half. Here, $y_1$ and $\alpha_1$ are again piecewise constant and close to each other like in test A.

In this test, we had to take $a_{Pr}^{-1}=0.1$.
For the QGD regularization, the quality of the solution is good and corresponds well to both the results of \cite{LA12} and \cite{ABR20}. However, the QHD regularization behaves worse in this case.
If we only double the value of $a$, the computations fail. After we both double $a$ and halve $\beta$, the computation is completed successfully. However, the quality of the velocity and the temperature plots worsens.

In the QGD case, if we take $a_S=0$, the computations run normally and quality of the solution is preserved.
\begin{figure*}[hbt!]
\center{\includegraphics[width=1\linewidth]{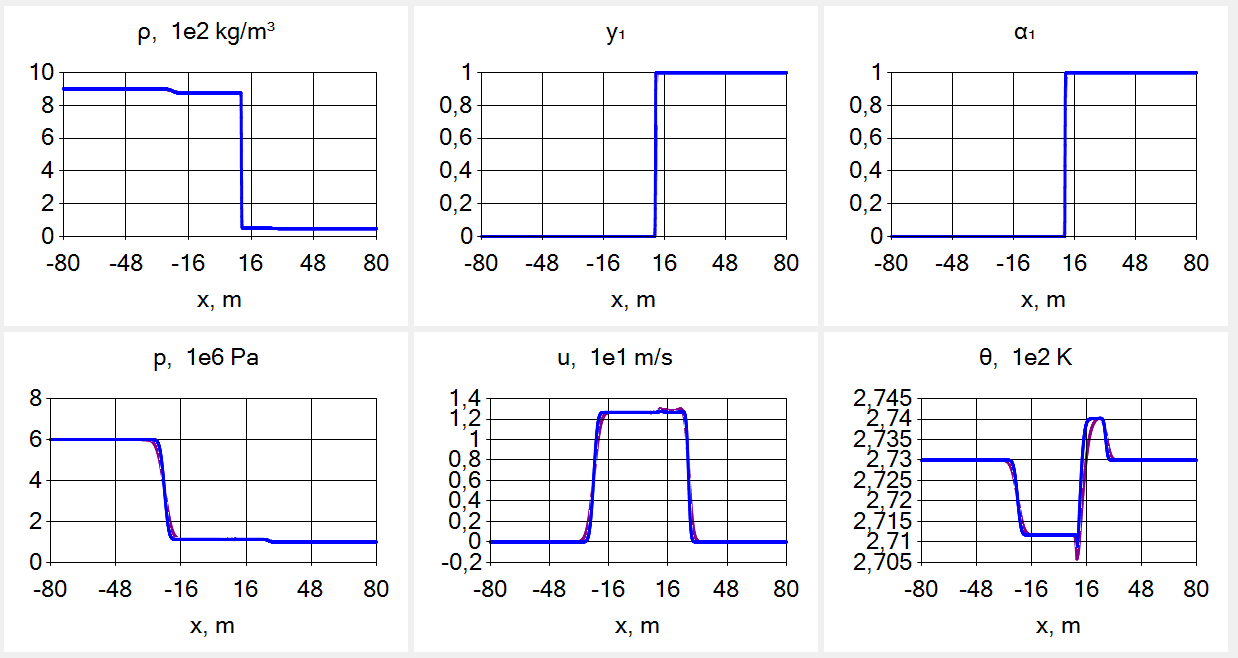}}
\caption{Numerical results for carbon dioxide depressurization (test G)
 for $N=1200$ (dark magenta), $4000$ (blue), $a=0.8$ and $\beta=0.1$ (the QGD regularization)}
\label{fig:abgrall}
\end{figure*}

\section{CONCLUSION}
In this paper, we have taken the four-equation model describing the dynamics of heterogeneous compressible binary mixtures in the case of the common velocity, temperature and pressure of the components, with the stiffened gas equations of state.
We have studied its known quasi-homogeneous form, with the excluded volume concentrations and the quadratic equation for the common pressure.
Namely, we have justified the correct choice of the physical root of the equation and presented two new rather simple expressions for the squared speed of sound and the related balance PDE for the pressure.
We have also compared several known formulas for the squared speed of sound in mixtures used in literature and found that the speed in the present model is minimal of them.

\par Next, the problem of constructing two regularizations of some well-known types for a heterogeneous mixture model has been solved for the first time by exploiting that quasi-homogeneous form. Some properties of these regularizations have been given too.
In the 1D case, new explicit two-level in time and symmetric three-point in space finite-difference schemes without limiters have been constructed based on the regularizations.
Numerical results for a number of test flows with shock waves known in literature have been given using these schemes.

\par The considered model for the heterogeneous binary mixtures can rather easily be generalized and successfully applied to important problems with phase transition\cite{LeMSN14,CBS17IJNMF,CBS17CF,ABR20,BCPCA22,DSPB22,P22}. The successful results of our numerical tests open the possibility to solve problems of this type and 2D and 3D problems in the frame of our regularization-based approach in future.

\begin{acknowledgments}
This study was supported by the Russian Science Foundation, grant no. 22-11-00126
(A.~Zlotnik, Sections II and III)
and
by the Moscow Center of Fundamental and Applied Mathematics Agreement with the Ministry of Science and Higher Education of the Russian Federation, grant no.~075-15-2022-283
(the both authors, Sections IV and V).
\end{acknowledgments}

\section*{Data Availability Statement}

The data that support the findings of this study are available from the corresponding author upon reasonable request.

\appendix
\section{Proofs of Propositions from Section \ref{Section II} }
\begin{proof}[Proof of Proposition \ref{the main}]
Eq. \eqref{eq for pk}
and formula \eqref{sg EOS cons1a} for $p$ imply that $p_k=p=p_+$ and
\begin{eqnarray}
 \sigma^{(k)}=\frac{\alpha_kR_kr_k}{c_V\rho}=\frac{\alpha_k(p_k+p_{*k})}{c_V\rho\theta}
\nonumber\\
 =\frac{\alpha_k(p_+ +p_{*k})}{c_V\rho\theta},\ \ k=1,2,
\label{form for sigma k}\\
 \langle\sigma^{(k)}\rangle
 =\frac{p_+ +\langle\alpha_kp_{*k}\rangle}{c_V\rho\theta}=\gamma-1.
\label{form for l sigma k r}
\end{eqnarray}
Using formula \eqref{sg EOS cons2a} for $\rho(\ve-\ve_0)$, we obtain
\begin{eqnarray*}
 b=\frac{p_+ +\langle\alpha_kp_{*k}\rangle}{c_V\rho\theta}(c_V\rho\theta+\langle\alpha_kp_{*k}\rangle)
 \\[1mm]
 -\frac{p_+\langle\alpha_kp_{*k}\rangle +\langle\alpha_kp_{*k}^2\rangle}{c_V\rho\theta}
 -\langle p_{*k}\rangle
\\[1mm]
 =p_+ +\frac{\langle\alpha_kp_{*k}\rangle^2-\langle\alpha_kp_{*k}^2\rangle}{c_V\rho\theta}
 -\langle(1-\alpha_k)p_{*k}\rangle.
\end{eqnarray*}
Taking into account the formula
\begin{eqnarray*}
 \langle\alpha_kp_{*k}\rangle^2-\langle\alpha_kp_{*k}^2\rangle
 =(\alpha_1^2-\alpha_1)p_{*1}^2
 +(\alpha_2^2-\alpha_2)p_{*2}^2
 \\[1mm]
 +2\alpha_1\alpha_2 p_{*1}p_{*2}
 =-\alpha_1\alpha_2\Delta_*^2,
\end{eqnarray*}
we derive the first formula \eqref{form for b and c}.

\par Formulas \eqref{sg EOS cons2a} and $\langle\alpha_k\rangle=1$ lead to the expressions
\begin{equation}
 \begin{array}{r}\rho(\ve-\ve_0)-p_{*1}=c_V\rho\theta+\alpha_2\Delta_*,
 \\[1mm]
 \rho(\ve-\ve_0)-p_{*2}=c_V\rho\theta-\alpha_1\Delta_*.\end{array}
\label{form for mult of ak}
\end{equation}
Therefore, the following formula for $c$ holds
\begin{eqnarray*}
 c=(\sigma^{(1)}p_{*2}+\sigma^{(2)}p_{*1})c_V\rho\theta
 \\[1mm]
 +(\alpha_2\sigma^{(1)}p_{*2}-\alpha_1\sigma^{(2)}p_{*1})\Delta_*-p_{*1}p_{*2}.
\end{eqnarray*}
Invoking formula \eqref{form for sigma k}, we can rewrite its terms as
\begin{eqnarray}
\sigma^{(1)}p_{*2}+\sigma^{(2)}p_{*1})c_V\rho\theta\nonumber\\[1mm]
 =p_+(\alpha_1p_{*2}+\alpha_2p_{*1})+p_{*1}p_{*2},
\label{si 1 p2 +si 2 p1}
\\[1mm]
 \alpha_2\sigma^{(1)}p_{*2}-\alpha_1\sigma^{(2)}p_{*1}\nonumber\\[1mm]
 =\frac{\alpha_1\alpha_2[(p_+ +p_{*1})p_{*2}-(p_+ +p_{*2})p_{*1}]}{c_V\rho\theta}
 \nonumber\\[1mm]
 =\frac{\alpha_1\alpha_2 p_+\Delta_*}{c_V\rho\theta}
\nonumber
\end{eqnarray}
and find that
\[
 c=p_+\Big(\alpha_1p_{*2}+\alpha_2p_{*1}
 +\frac{\alpha_1\alpha_2}{c_V\rho\theta}\Delta_*^2\Big)=-p_+p_-.
\]
The rest of the Proposition is obvious.
\end{proof}
\begin{proof}[Proof of Proposition \ref{lem: d greater 0}]
The proof of formula \eqref{second form for d} is absent in \cite{LeMSN14}, so we give it for completeness.
Note that
\begin{eqnarray*}
 (b_1\pm b_2)^2=(a_1\pm a_2-(p_{*1}\pm p_{*2}))^2
 \\[1mm]
 =a_1^2+a_2^2-2(a_1p_{*1}+a_2p_{*2})+p_{*1}^2+p_{*2}^2\pm 2(a_1a_2-c)
\end{eqnarray*}
since $c=a_1p_{*2}+a_2p_{*1}-p_{*1}p_{*2}$.
Consequently, we have
\begin{eqnarray*}
d=(b_1+b_2)^2+4c\\[1mm]
=(b_1-b_2)^2+4(a_1a_2-c)+4c=(b_1-b_2)^2+4a_1a_2.
\end{eqnarray*}

\par Next, due to this formula and formulas \eqref{form for mult of ak}, setting $a:=c_V\rho\theta$, we obtain
\begin{eqnarray*}
 d=(\sigma^{(2)}(a-\alpha_1\Delta_*)-\sigma^{(1)}(a+\alpha_2\Delta_*)-\Delta_*)^2\\[1mm]
 +4\sigma^{(1)}\sigma^{(2)}(a+\alpha_2\Delta_*)(a-\alpha_1\Delta_*)
\\[1mm]
 =(\sigma^{(2)}(a-\alpha_1\Delta_*))^2+(\sigma^{(1)}(a+\alpha_2\Delta_*))^2\\[1mm]
 +2\sigma^{(1)}\sigma^{(2)}(a+\alpha_2\Delta_*)(a-\alpha_1\Delta_*)
\\[1mm]
 -2(\sigma^{(2)}(a-\alpha_1\Delta_*)-\sigma^{(1)}(a+\alpha_2\Delta_*))\Delta_*
 +\Delta_*^2
\\[1mm]
=\big[(\alpha_1\sigma^{(2)}-\alpha_2\sigma^{(1)})^2
+2(\alpha_1\sigma^{(2)}+\alpha_2\sigma^{(1)})+1\big]\Delta_*^2
\\[1mm]
+2a\big[-\alpha_1(\sigma^{(2)})^2+\alpha_2(\sigma^{(1)})^2
\\[1mm]
+(\alpha_2-\alpha_1)\sigma^{(1)}\sigma^{(2)}
+\sigma^{(1)}-\sigma^{(2)}\big]\Delta_*
+a^2(\sigma^{(1)}+\sigma^{(2)})^2.
\end{eqnarray*}
Since $\sigma^{(1)}+\sigma^{(2)}=\gamma-1$, see \eqref{rho sigmak}, we derive equality in \eqref{third form for d}.

\par The discriminant $d_0$ of the derived quadratic polynomial with respect to $\Delta_*$ is such that
\begin{eqnarray*}
\frac{1}{4a^2}d_0:=\big[(\alpha_2\sigma^{(1)}-\alpha_1\sigma^{(2)})(\sigma^{(1)}+\sigma^{(2)})+\sigma^{(1)}-\sigma^{(2)}\big]^2
\\[1mm]
-\big[(\alpha_1\sigma^{(2)}-\alpha_2\sigma^{(1)})^2
+2(\alpha_1\sigma^{(2)}+\alpha_2\sigma^{(1)})+1\big]\\[1mm]
\times(\sigma^{(1)}+\sigma^{(2)})^2
=2\big[(\alpha_2\sigma^{(1)}-\alpha_1\sigma^{(2)})((\sigma^{(1)})^2-(\sigma^{(2)})^2)
\\[1mm]
-(\alpha_1\sigma^{(2)}+\alpha_2\sigma^{(1)})(\sigma^{(1)}+\sigma^{(2)})^2\big]
-4\sigma^{(1)}\sigma^{(2)}
\\[1mm]
=-4\big[(\alpha_2 +\alpha_1)\sigma^{(1)}\sigma^{(2)}(\sigma^{(1)}+\sigma^{(2)})
+\sigma^{(1)}\sigma^{(2)})
\\[1mm]
=-4\sigma^{(1)}\sigma^{(2)}((\sigma^{(1)}+\sigma^{(2)}+1)<0.
\end{eqnarray*}
This proves the property $d>0$ once more independently of Proposition \ref{the main}.
\end{proof}

\par Now recall that, for any constant $a_k$ and $c_k$, $k=1,2$ (except for $c_1=c_2=0$), we get
\begin{eqnarray}
(\partial_t+\*u\cdot\nabla)\frac{\langle a_k\rho_k\rangle}{\langle c_k\rho_k\rangle}
 =\frac{1}{\langle c_k\rho_k\rangle^2}
 [\langle c_k\rho_k\rangle\langle a_k(\partial_t+\*u\cdot\nabla)\rho_k\rangle
 \nonumber\\[1mm]
 -\langle a_k\rho_k\rangle\langle c_k(\partial_t+\*u\cdot\nabla)\rho_k\rangle]=0
\label{der of rat fun}
\end{eqnarray}
provided that $\langle c_k\rho_k\rangle\neq 0$,
since $(\partial_t+\*u\cdot\nabla)\rho_k=-\rho_k\dv\*u$ due to the mass balance PDEs~\eqref{mass eq alpha NS}.
\begin{proof}[Proof of Proposition \ref{prop:speed of sound}]
We differentiate Eq. \eqref{quad eq for p} for $p=p_+$ and constant $\ve_0$ and $\sigma^{(k)}$, $k=1,2$, and get
\begin{eqnarray*}
\Big(\partial_\rho p_+ +\frac{p_+}{\rho^2}\partial_\ve p_+\Big)(2p_+-b)
\\[1mm]
=\Big(\partial_\rho b+\frac{p_+}{\rho^2}\partial_\ve b\Big)p_++\partial_\rho c+\frac{p_+}{\rho^2}\partial_\ve c
\\[1mm]
=\langle\sigma^{(k)}\rangle\Big(\ve-\ve_0+\frac{p_+}{\rho}\Big) p_+\\[1mm]
+(\sigma^{(1)}p_{*2}+\sigma^{(2)}p_{*1})\Big(\ve-\ve_0+\frac{p_+}{\rho}\Big).
\end{eqnarray*}
This formula leads to an intermediate formula
\begin{eqnarray}
 c_s^2=\frac{\langle\sigma^{(k)}\rangle p_+ +\sigma^{(1)}p_{*2}+\sigma^{(2)}p_{*1}}{2p_+-b}\frac{\rho(\ve-\ve_0)+p_+}{\rho}.
\label{form for cs2 aux}
\end{eqnarray}

\par Next, formulas \eqref{sg EOS cons2b}, \eqref{form for l sigma k r} and \eqref{si 1 p2 +si 2 p1}
allows us to write finally
\begin{eqnarray}
 c_s^2
\nonumber\\[1mm]
 =\frac{(p_+ +\langle\alpha_kp_{*k}\rangle)p_+ +p_+(\alpha_1p_{*2}+\alpha_2p_{*1})+p_{*1}p_{*2}}
 {c_V\rho\theta(2p_+-b)}
\nonumber\\[1mm]
\times\gamma c_V\theta
 =\frac{(p_+ +p_{*1})(p_+ +p_{*2})}{\rho\sqrt{d}}\gamma.
\label{form for cs2 aux 2}
\end{eqnarray}

\par Differentiating Eq. \eqref{quad eq for p} for $p=p_+$ again, we can write
\begin{eqnarray}
(\partial_tp_+ +\*u\cdot\nabla p_+)(2p_+-b)
\nonumber\\[1mm]
=(\partial_tb+\*u\cdot\nabla b)p_+ +\partial_tc+\*u\cdot\nabla c
\nonumber\\[1mm]
=[\partial_t(\rho(\ve-\ve_0))+\*u\cdot\nabla(\rho(\ve-\ve_0))]
\nonumber\\[1mm]
\times(\langle\sigma^{(k)}\rangle p_+
+\sigma^{(1)}p_{*2}+\sigma^{(2)}p_{*1}),
\label{aux eq for p_+}
\end{eqnarray}
where we have applied the auxiliary equations
\[
(\partial_t+\*u\cdot\nabla)\sigma^{(k)}=0,\ \ k=1,2,\ \ (\partial_t+\*u\cdot\nabla)\gamma=0
\]
following from Eq. \eqref{der of rat fun}.
Next, the balance PDE for the internal energy \eqref{int en homog} implies
\begin{eqnarray*}
\partial_t(\rho(\ve-\ve_0))+\*u\cdot\nabla(\rho(\ve-\ve_0))
\\[1mm]
=\partial_t(\rho(\ve-\ve_0))+\dv(\rho(\ve-\ve_0)\*u)-\rho(\ve-\ve_0)\dv\*u
\\[1mm]
=\partial_t(\rho\ve)+\dv(\rho\ve\*u)-\rho(\ve-\ve_0)\dv\*u
\\[1mm]
=-(\rho(\ve-\ve_0)+p_+)\dv\*u+\dv(\vk\nabla\theta)+\Pi:\nabla\*u+Q,
\end{eqnarray*}
where the following PDE
\[
 \partial_t(\rho\ve_0)+\dv(\rho\ve_0\*u)
=\langle(\partial_t\rho_k+\dv(\rho_k\*u))\ve_{0k}\rangle=0
\]
has been applied, see the mass balance PDEs \eqref{mass eq alpha NS}.
Now from Eq. \eqref{aux eq for p_+} due to
formulas \eqref{form for cs2 aux} and \eqref{sg EOS cons2b}, we finally get
\begin{eqnarray*}
\partial_tp_+ +\*u\cdot\nabla p_+
\\[1mm]
=-\rho c_s^2\dv\*u+\frac{c_s^2}{\gamma c_V\theta}(\dv(\vk\nabla\theta)+\Pi:\nabla\*u+Q),
\end{eqnarray*}
and the proof is complete.
\end{proof}
\begin{proof}[Proof of Proposition \ref{prop:speed of sound A}]
The quadratic equation \eqref{quad eq for p} implies
the formula for the differential of $p_+$:
\[
(2p_+ -b)dp_+=p_+ db+dc,
\]
where due to formulas \eqref{coeff b}, $\langle\sigma^{(k)}\rangle=\gamma-1$ and \eqref{coeff c} we get
\begin{eqnarray*}
db=\rho(\ve-\ve_0)d\gamma
+(\gamma-1)d(\rho\ve-\langle\ve_{0k}d\rho_k\rangle)
-\langle p_{*k}d\sigma^{(k)}\rangle
\nonumber\\
=\langle (\rho(\ve-\ve_0)-p_{*k})d\sigma^{(k)}\rangle
+(\gamma-1)(-\langle\ve_{0k}d\rho_k\rangle
+d(\rho\ve)),
\nonumber\\
dc=\rho(\ve-\ve_0)(p_{*2}d\sigma^{(1)}+p_{*1}d\sigma^{(2)})
\nonumber\\
+(\sigma^{(1)}p_{*2}+\sigma^{(2)}p_{*1})(d(\rho\ve)-\langle\ve_{0k}d\rho_k\rangle)
-p_{*1}p_{*2}d\gamma
\nonumber\\
=(\rho(\ve-\ve_0)-p_{*1})p_{*2}d\sigma^{(1)}
+(\rho(\ve-\ve_0)-p_{*2})p_{*1}d\sigma^{(2)}
\nonumber\\
+(\sigma^{(1)}p_{*2}+\sigma^{(2)}p_{*1})(-\langle\ve_{0k}d\rho_k\rangle+d(\rho\ve)).
\end{eqnarray*}
Consequently, we obtain
\begin{eqnarray*}
p_+ db+dc
\\
=\mathcal{H}_1d\sigma^{(1)}
+\mathcal{H}_2d\sigma^{(2)}
+\sqrt{d}\mathcal{P}(-\langle\ve_{0k}d\rho_k\rangle+d(\rho\ve))
\end{eqnarray*}
with the  functions
$\mathcal{H}_1$, $\mathcal{H}_2$ and $\mathcal{P}$ defined in the statement of Proposition \ref{prop:speed of sound A}.
Next, the straightforward calculation gives
\begin{eqnarray*}
d\sigma^{(1)}=\frac{R_1c_{V2}}{(c_V\rho)^2}(\rho_2d\rho_1-\rho_1d\rho_2),
\nonumber\\
d\sigma^{(2)}=-\frac{R_2c_{V1}}{(c_V\rho)^2}(\rho_2d\rho_1-\rho_1d\rho_2).
\end{eqnarray*}
Inserting these formulas in the previous one for $p_+ db+dc$, we derive formula \eqref{form for dp} for $dp_+$ with the functions
$\mathcal{P}_1$ and $\mathcal{P}_2$ defined in the statement of Proposition \ref{prop:speed of sound A}.

\par Now we can calculate
\begin{eqnarray*}
\Big\langle\frac{\rho_k}{\rho}\mathcal{P}_k\Big\rangle
+\frac{\rho\ve+p_+}{\rho}\mathcal{P}
=-\mathcal{P}\Big\langle\frac{\rho_k\ve_{0k}}{\rho}\Big\rangle+\frac{\rho\ve+p_+}{\rho}\mathcal{P}
\\
=\frac{(\gamma-1)p_+
+\sigma^{(1)}p_{*2}+\sigma^{(2)}p_{*1}}{\sqrt{d}}\frac{\rho(\ve-\ve_0)+p_+}{\rho}.
\end{eqnarray*}
Applying formulas \eqref{form for cs2 aux} and \eqref{form for cs2 aux 2}, we complete the proof.
\end{proof}
\begin{proof}[Proof of Proposition \ref{prop: 2nd formula for c_s^2}]
We differentiate the rational equation \eqref{orig eq for p} for $p=p_+$
under the previous assumption that $\ve_0$ and $\sigma^{(k)}$, $k=1,2$, are constant in \eqref{coeff b}-\eqref{coeff c}:
\begin{eqnarray*}
-\Big\langle\frac{\sigma^{(k)}(\rho(\ve-\ve_0)-p_{*k})}{(p_++p_{*k})^2}\Big\rangle\Big(\partial_\rho p_+ +\frac{p_+}{\rho^2}\partial_\ve p_+\Big)
\\
+\Big\langle\frac{\sigma^{(k)}(\ve-\ve_0)}{p_++p_{*k}}\Big\rangle
+\frac{p_+}{\rho^2}\Big\langle\frac{\sigma^{(k)}\rho}{p_++p_{*k}}\Big\rangle=0.
\end{eqnarray*}
Consequently, we first get
\begin{eqnarray*}
c_s^2=\partial_\rho p_++\frac{p_+}{\rho^2}\partial_\ve p_+
\\
=\Big\langle\frac{\sigma^{(k)}\big(\ve-\ve_0+\frac{p_+}{\rho}\big)}{p_++p_{*k}}\Big\rangle
\Big\langle\frac{\sigma^{(k)}(\rho(\ve-\ve_0)-p_{*k})}{(p_++p_{*k})^2}\Big\rangle^{-1}.
\end{eqnarray*}
The first term of the last expression is simplified as follows
\begin{eqnarray*}
\Big\langle\frac{\sigma^{(k)}\big(\ve-\ve_0+\frac{p_+}{\rho}\big)}{p_++p_{*k}}\Big\rangle
\\
=\frac{1}{\rho}\Big\langle\frac{\sigma^{(k)}(\rho(\ve-\ve_0)-p_{*k}+p+p_{*k})}{p_++p_{*k}}\Big\rangle
\\
=\frac{1}{\rho}(1+\langle\sigma^{(k)}\rangle)
=\frac{\gamma}{\rho}
\end{eqnarray*}
according to the rational equation \eqref{orig eq for p}
and the third relation \eqref{rho sigmak}, and the proof is complete.
\end{proof}

\par To show explicitly that two such different formulas for $c_s^2$ derived in Propositions \ref{prop:speed of sound} and \ref{prop: 2nd formula for c_s^2} coincide, we recall quantities $a_k$ and $b_k$ introduced in Proposition \ref{lem: d greater 0} and perform the following transformations
\begin{eqnarray*}
\Big\langle\frac{\sigma^{(k)}(\rho(\ve-\ve_0)-p_{*k})}{(p_+ +p_{*k})^2}\Big\rangle(p_+ +p_{*1})(p_+ +p_{*2})
\\
=\frac{a_1}{p_+ +p_{*1}}(p_+ +p_{*2})
+\frac{a_2}{p_+ +p_{*2}}(p_+ +p_{*1})
\\
=\Big(1-\frac{a_2}{p_+ +p_{*2}}\Big)(p_+ +p_{*2})
+\Big(1-\frac{a_1}{p_+ +p_{*1}}\Big)(p_+ +p_{*1})
\\
=2p_+ +p_{*1}+p_{*2}-a_1-a_2
\\
=2p_+ -(b_1+b_2)=2p_+ -b=\sqrt{d}.
\end{eqnarray*}
Here we have applied twice the rational equation  \eqref{orig eq for p} rewritten in the short form
$
 \big\langle\frac{a_k}{p +p_{*k}}\big\rangle=1.
$
\begin{proof}[Proof of Proposition \ref{prop: two speeds of sound}]
Concerning formula \eqref{cs cw frac}, see Proposition 6 in \cite{FMM10}.
But, within the framework of the quasi-homogeneous model, we find it important to present another proof based on the above formulas for $p_\pm$ and $c_s^2$.

First, we are going to verify that the first general formula \eqref{cp zeta} for $\zeta_k$ (see formula (102) in \cite{FMM10}) implies the second particular one.
Due to formulas \eqref{sg EOS}, we get
\begin{eqnarray}
 r_k=\frac{p_k+p_{*k}}{R_k\theta},\ \
  \ve_k(\theta,p_k)=c_{Vk}\theta+\frac{R_k p_{*k}\theta}{p_k+p_{*k}}+\ve_{0k},
\label{rk and vek in theta and pk}
\end{eqnarray}
for $k=1,2$.
Therefore, since $c_{pk}=c_{Vk}+R_k$, we can write
\begin{eqnarray*}
 \zeta_k=\Big(1-\frac{c_{Vk}p_k+(c_{Vk}+R_k)p_{*k}}{c_{pk}(p_k+p_{*k})}\Big)\frac{\theta}{p_k}
 \\[1mm]
 =\frac{R_kp_k}{c_{pk}(p_k+p_{*k})}\frac{\theta}{p_k}
 =\frac{1}{c_{pk}r_k},\ \ k=1,2.
\end{eqnarray*}

\par Next, we denote by $\mathcal{F}_1$ and $\mathcal{F}_2$ the first and second fractions on the right in formula \eqref{cs cw frac}.
Representing $c_{sW}^2$ in terms of $c_{pk}r_k$, we obtain the following formulas for the fractions
\begin{eqnarray*}
 \mathcal{F}_1
 :=\frac{1}{\rho c_{sW}^2}=\frac{\alpha_2(\gamma_1-1)c_{p1}r_1+\alpha_1(\gamma_2-1)c_{p2}r_2}
                  {\theta(\gamma_1-1)c_{p1}r_1(\gamma_2-1)c_{p2}r_2},
\\[1mm]
\mathcal{F}_2=\frac{\alpha_1\alpha_2(c_{p1}r_1-c_{p2}r_2)^2}{\rho c_p\theta c_{p1}r_1c_{p2}r_2}.
\end{eqnarray*}
We rewrite the numerator of $\mathcal{F}_2$ as follows
\begin{eqnarray*}
\alpha_1\alpha_2(c_{p1}r_1-c_{p2}r_2)^2=
\alpha_1\alpha_2[(c_{p1}r_1)^2+(c_{p2}r_2)^2]\\[1mm]
+(1-2\alpha_1\alpha_2)c_{p1}r_1c_{p2}r_2
-c_{p1}r_1c_{p2}r_2
\\[1mm]
=(\alpha_2c_{p1}r_1+\alpha_1c_{p2}r_2)(\alpha_1c_{p1}r_1+\alpha_2c_{p2}r_2)-c_{p1}r_1c_{p2}r_2
\end{eqnarray*}
due to the formulas $1-2\alpha_1\alpha_2=(\alpha_1+\alpha_2)^2-2\alpha_1\alpha_2=\alpha_1^2+\alpha_2^2$.

\par Then, applying the formula $\alpha_1c_{p1}r_1+\alpha_2 c_{p2}r_2=\rho c_p$, we obtain
\begin{eqnarray*}
 \mathcal{F}_1+\mathcal{F}_2
 =\frac{1}{\theta(\gamma_1-1)(\gamma_2-1)c_{p1}c_{p2}r_1 r_2}
\\[1mm]
 \times[\alpha_2(\gamma_1-1)c_{p1}r_1+\alpha_1(\gamma_2-1)c_{p2}r_2
\\[1mm]
 +(\gamma_1-1)(\gamma_2-1)(\alpha_2c_{p1}r_1+\alpha_1c_{p2}r_2)]
 -\frac{1}{\rho c_p\theta}
\\[1mm]
 =\frac{1}{\theta(\gamma_1-1)(\gamma_2-1)c_{p1}c_{p2}r_1 r_2}\\[1mm]
 \times[\alpha_2\gamma_2(\gamma_1-1)c_{p1}r_1+\alpha_1\gamma_1(\gamma_2-1)c_{p2}r_2]
 -\frac{1}{\rho c_p\theta}.
\end{eqnarray*}
The formulas $(\gamma_k-1)c_{pk}=\gamma_kR_k$, $k=1,2$, and $c_p=\gamma c_V$
and reduction by $\gamma_1\gamma_2$ lead to the formula
\begin{eqnarray*}
\mathcal{F}_1+\mathcal{F}_2
 =\frac{\alpha_2R_1r_1+\alpha_1R_2r_2}{\theta R_1r_1R_2r_2}-\frac{1}{\gamma c_V\rho\theta}.
\end{eqnarray*}
Due to the first formula \eqref{rk and vek in theta and pk}, one can pass from $R_kr_k$ to $p_k=p$:
\begin{eqnarray*}
\mathcal{F}_1+\mathcal{F}_2
 =\frac{\alpha_2(p+p_{*1})+\alpha_1(p+p_{*2})}{(p+p_{*1})(p+p_{*2})}-\frac{1}{\gamma c_V\rho\theta}
\\[1mm]
 =\frac{\gamma(p+\alpha_1p_{*2}+\alpha_2p_{*1})
 -(p+p_{*1})(p+p_{*2})\frac{1}{c_V\rho\theta}}{\gamma(p+p_{*1})(p+p_{*2})}.
\end{eqnarray*}

\par We need the formulas
\begin{eqnarray*}
(p+\alpha_1p_{*2}+\alpha_2p_{*1})(p+\langle\alpha_kp_{*k}\rangle)\\[1mm]
=p^2+(p_{*1}+p_{*2})p+(\alpha_1p_{*2}+\alpha_2 p_{*1})\langle\alpha_kp_{*k}\rangle,
\\[1mm]
\alpha_1p_{*2}+\alpha_2p_{*1})\langle\alpha_kp_{*k}\rangle-p_{*1}p_{*2}
 \\[1mm]
 =\alpha_1\alpha_2(p_{*1}^2+p_{*2}^2)+(\alpha_1^2+\alpha_2^2-(\alpha_1+\alpha_2)^2)p_{*1}p_{*2}
\\[1mm]
 =\alpha_1\alpha_2(p_{*1}-p_{*2})^2.
\end{eqnarray*}
Applying them together with formula \eqref{sg EOS cons1a}, we have
\begin{eqnarray*}
 (p+p_{*1})(p+p_{*2})\\[1mm]
 =(p+\alpha_1p_{*2}+\alpha_2p_{*1})(p+\langle\alpha_kp_{*k}\rangle)-\alpha_1\alpha_2\Delta_*^2
\\[1mm]
 =(p+\alpha_1p_{*2}+\alpha_2p_{*1})(\gamma-1)c_V\rho\theta-\alpha_1\alpha_2\Delta_*^2,
\end{eqnarray*}
and finally we can represent $\mathcal{F}_1+\mathcal{F}_2$ as
\begin{eqnarray*}
 \mathcal{F}_1+\mathcal{F}_2=\frac{1}{\gamma(p+p_{*1})(p+p_{*2})}
 \\[1mm]
 \times\Big\{\gamma(p+\alpha_1 p_{*2}+\alpha_2p_{*1})
 \\[1mm]
 -\Big[(\gamma-1)(p+\alpha_1p_{*2}+\alpha_2p_{*1})
 -\frac{\alpha_1\alpha_2\Delta_*^2}{c_V\rho\theta}\Big]\Big\}
 \\[1mm]
 =\frac{p-p_-}{\gamma(p+p_{*1})(p+p_{*2})}=\frac{1}{\rho c_s^2},
\end{eqnarray*}
see definition \eqref{form for pp and pm} of $p_-$ and formula \eqref{for speed of sound} for $c_s^2$, with $\sqrt{d}=p_+-p_-$.
Formula \eqref{cs cw frac} is proved.
\end{proof}
\begin{proof}[Proof of Proposition \ref{prop: 3 sq speeds of sound}]
We first notice that
$\rho\sqrt{d}(\gamma-1)c_V\theta=(p_+-p_-)R\rho\theta$.
Therefore, applying also formula \eqref{for speed of sound} for $c_s^2$,
the equality of two numerators in \eqref{form for cs2 aux 2} and definition \eqref{form for pp and pm} of $p_+$, we can write
\begin{eqnarray*}
 \rho\sqrt{d}(c_s^2-\gamma(\gamma-1)c_V\theta)
 \\[1mm]
 =\gamma[R\rho\theta p_+ +p_+(\alpha_1p_{*2}+\alpha_2p_{*1})+p_{*1}p_{*2}
\\[1mm]
 -(p_+-p_-)R\rho\theta]
\\[1mm]
=\gamma[(R\rho\theta+\langle\alpha_kp_{*k}\rangle)(\alpha_1p_{*2}+\alpha_2p_{*1})+p_{*1}p_{*2}+p_-R\rho\theta]
\\[1mm]
=\gamma[R\rho\theta(\alpha_1p_{*2}+\alpha_2p_{*1}+p_-)
\\[1mm]
 -\langle\alpha_kp_{*k}\rangle(\alpha_1p_{*2}+\alpha_2p_{*1})+p_{*1}p_{*2}].
\end{eqnarray*}
Using definition \eqref{form for pp and pm} of $p_-$ and the formula
\begin{eqnarray*}
\langle\alpha_kp_{*k}\rangle(\alpha_1p_{*2}+\alpha_2p_{*1})-p_{*1}p_{*2}\\[1mm]
 =\alpha_1\alpha_2(p_{*1}^2+p_{*2}^2)+(\alpha_1^2+\alpha_2^2-1)p_{*1}p_{*2}
\\[1mm]
 =\alpha_1\alpha_2(p_{*1}-p_{*2})^2,
\end{eqnarray*}
we further derive the first inequality \eqref{ineq for cs}:
\begin{eqnarray*}
 \rho\sqrt{d}(c_s^2-\gamma(\gamma-1)c_V\theta)
 =\gamma\Big(-\frac{R}{c_V}\alpha_1\alpha_2\Delta_*^2-\alpha_1\alpha_2\Delta_*^2\Big)
 \\[1mm]
 =-\gamma^2\alpha_1\alpha_2\Delta_*^2\leq 0.
\end{eqnarray*}

\par The second inequality \eqref{ineq for cs} has recently been proved in Proposition 1 in \cite{ZL23}.
\end{proof}


\section*{References}

\begin{thebibliography}{999}
\bibitem{FMM10}
T. Fl\"{a}tten, A. Morin, and S. T. Munkejord, ``Wave propagation in multicomponent flow models,''
\emph{SIAM J. Appl. Math.} \textbf{70}, 2861--2882 (2010). https://doi.org/10.1137/090777700
\bibitem{FL11}
T. Fl\"{a}tten and H. Lund, ``Relaxation two-phase models and the subcharacteristic condition,''
\emph{Math. Models Meth. Appl. Sci.} \textbf{21}, 2379--2407 (2011).
https://doi.org/10.1142/S0218202511005775
\bibitem{ZMWS22}
C. Zhang, I. Menshov, L. Wang, and Z. Shen, ``Diffuse interface relaxation model for two-phase compressible flows with diffusion processes,'' \emph{J. Comput. Phys.} \textbf{466}, article 111356 (2022).
https://doi.org/10.1016/j.jcp.2022.111356

\bibitem{LeMSN14}
S. Le~Martelot, R. Saurel, and B. Nkonga,
``Towards the direct numerical simulation of nucleate boiling flows,''
\emph{Int. J. Multiphase Flow} \textbf{66}, 62--78 (2014).
\allowbreak
http://dx.doi.org/10.1016/j.ijmultiphaseflow.2014.06.010
\bibitem{CBS17IJNMF}
A. Chiapolino, P. Boivin, and R. Saurel, ``A simple phase transfer relaxation solver for liquid--vapor flows,''
\emph{Int. J. Numer. Meth. Fluids} \textbf{83}, 583--605 (2017).
http://dx.doi.org/10.1002/fld.4282
\bibitem{CBS17CF}
A. Chiapolino, P. Boivin, and R. Saurel, ``A simple and fast phase transfer relaxation
solver for compressible multicomponent two-phase flows,'' \emph{Comput. Fluids} \textbf{150}, 31--45 (2017). https://doi.org/10.1016/j.compfluid.2017.03.022
\bibitem{ABR20}
R. Abgrall, P. Bacigaluppi, and B. Re, ``On the simulation of multicomponent and multiphase compressible flows,''
\emph{ERCOFTAC Bulletin} \textbf{124} (2020).
\bibitem{BCPCA22}
P. Bacigaluppi, J. Carlier, M. Pelanti, P. M. Congedo, and R. Abgrall,
``Assessment of a non-conservative four-equation multiphase system with phase transfer,''
\emph{J. Sci. Comput.} \textbf{90}, article 28 (2022).
https://doi.org/10.1007/s10915-021-01706-6
\bibitem{DSPB22}
A.D. Demou, N. Scapin, M. Pelanti, and L. Brandt,
``A pressure-based diffuse interface method for low-Mach multiphase flows with mass transfer,''
\emph{J. Comput. Phys.} \textbf{448}, article 110730 (2022).
\bibitem{P22}
M. Pelanti,
``Arbitrary-rate relaxation techniques for the numerical modeling of compressible two-phase flows with heat and mass transfer,''
\emph{Int. J. Multiphase Flow} \textbf{153}, article 104097 (2022).
\bibitem{Ch04}
B. N. Chetverushkin, \textit{Kinetic Schemes and Quasi-Gas Dynamic System of Equations} (CIMNE: Barcelona, 2008).
\bibitem{E07}
T. G. Elizarova, \textit{Quasi-Gas Dynamic Equations} (Springer: Berlin, 2009).
https://doi.org/10.1007/978-3-642-00292-2
%
\bibitem{EZCh14}
T. G. Elizarova, A. A. Zlotnik, and B. N. Chetverushkin, ``On quasi-gasdynamic and quasi-\-hyd\-rodynamic equations for binary mixtures of gases,''
\emph{Dokl. Math.} \textbf{90}, 1--5 (2014). https://doi.org/10.1134/s0965542519110058
%
\bibitem{KKPP18}
T. Kudryashova, Yu. Karamzin, V. Podryga and S. Polyakov.
``Two-scale computation of N2--H2 jet flow based on QGD and MMD on heterogeneous multi-core hardware,''
\emph{Adv. Eng. Software} \textbf{120}, 79--87 (2018).
https://doi.org/10.1016/j.advengsoft.2016.02.005.

\bibitem{BS18}
V. A. Balashov and E. B. Savenkov, ``Quasi-hydrodynamic model of multiphase fluid flows taking into account phase interaction,''
\emph{J. Appl. Mech. Tech. Phys.} \textbf{59}, 434--444 (2018).
https://doi.org/10.1134/S0021894418030069
\bibitem{BZ21}
V. Balashov and A. Zlotnik, ``On a new spatial discretization for a regularized 3D compressible isothermal Navier--Stokes--Cahn--Hilliard system of equations with boundary conditions,''
\emph{J. Sci. Comput.} \textbf{86},  article 33 (2021).
https://doi.org/10.1007/s10915-020-01388-6.
\bibitem{EZSh19}
T. G. Elizarova, A. A. Zlotnik, and E. V. Shil'nikov, ``Regularized equations for numerical simulation of flows of homogeneous binary mixtures of viscous compressible gases,''
\emph{Comput. Math. Math. Phys.} \textbf{59}, 1832--1847 (2019).
https://doi.org/10.1134/S0965542519110058
\bibitem{ZFL22}
A. Zlotnik, A. Fedchenko, and T. Lomonosov, ``Entropy correct spatial discretizations for 1D regularized systems of equations for gas mixture dynamics,''
\emph{Symmetry} \textbf{14}, article 2171 (2022).

\bibitem{ESh22}
T. G. Elizarova and E. V. Shil'nikov,
``Quasi-gasdynamic model and numerical algorithm for describing mixtures of different fluids,''
\emph{Comput. Math. Math. Phys.} \textbf{63}, (2023), 1319--1331.
https://doi.org/10.1134/S0965542523070059

https://doi.org/10.3390/sym14102171
\bibitem{ZL23}
A. Zlotnik and T. Lomonosov, ``On regularized systems of equations for gas mixture dynamics with new regularizing velocities and diffusion fluxes,'' \emph{Entropy} \textbf{25}, article 158 (2023).
https://doi.org/10.3390/e25010158

\bibitem{ZF22MMAS}
A. Zlotnik and A. Fedchenko, ``On properties of aggregated regularized systems of equations for a homogeneous multicomponent gas mixture,''
\emph{Math. Meth. Appl. Sci.} \textbf{45}, 8906--8927 (2022).
https://doi.org/10.1002/mma.8214

\bibitem{Z12MM2}
A. A. Zlotnik, ``On construction of quasi-gasdynamic systems of equations and the barotropic system with the potential body force,'' \emph{Math. Model.} \textbf{24}(4), 65--79 (2012). (In Russian).

\bibitem{GPT16}
J.-L. Guermond, B. Popov, and V. Tomov, ``Entropy viscosity method for the single material Euler equations in Lagrangian frame,''
\emph{Comput. Meth. Appl. Mech. Eng.} \textbf{300}, 402--426 (2016).
https://doi.org/10.1016/j.cma.2015.11.009
\bibitem{FL_MM20}
 E. Feireisl, M. Luk\'{a}\v{c}ov\'{a}-Medvidov\'{a}, and H. Mizerov\'{a},
``A finite volume scheme for the Euler system inspired by the two velocities approach,''
\emph{Numer. Math.} \textbf{144}, 89--132 (2020).
https://doi.org/10.1007/s00211-019-01078-y
\bibitem{DS21}
V. Dolej\v{s}\'{i} and M. Sv\"{a}rd,
``Numerical study of two models for viscous compressible fluid flows,''
\emph{J. Comput. Phys.} \textbf{427}, article 110068 (2021).
https://doi.org/10.1016/j.jcp.2020.110068

\bibitem{Z12CMMP}
A. A. Zlotnik, ``Spatial discretization of the one-dimensional quasi-gas\-dy\-na\-mic system of equations and the entropy balance equation,''
\emph{Comput. Math. Math. Phys.} \textbf{52}, 1060--1071 (2012).
https://doi.org/10.1134/S0965542512070111

\bibitem{KLC14}
K. Kitamura, M.-S. Liou, and C.-H. Chang, ``Extension and comparative study of AUSM-family schemes for compressible multiphase flow simulations,''
\emph{Commun. Comput. Phys.} \textbf{16}, 632--674 (2014).
https://doi.org/10.4208/cicp.020813.190214a
\bibitem{LA12}
H. Lund and P. Aursand,
``Two-phase flow of CO2 with phase transfer,'' \emph{Energy Procedia} \textbf{23}, 246--255 (2012).
https://doi.org/10.1016/j.egypro.2012.06.034
\bibitem{LF11}
Q. Li and S. Fu,
``A gas-kinetic BGK scheme for gas-water flow,''
\emph{Comput. Math. Appl.}  \textbf{61}, 3639--3652 (2011).
\bibitem{YC13}
G.-S. Yeom and K. S. Chang,
``A modified HLLC-type Riemann solver for the compressible six-equation two-fluid model,''
\emph{Comput. Fluids}  \textbf{61}, 3639--3652 (2011).
\bibitem{A99}
R. Abgrall, An extension of Roe's upwind scheme to algebraic equilibrium real gas models,
\emph{Comput. Fluids} 19, 171--182 (1991). https://doi.org/10.1016/0045-7930(91)90032-D
\bibitem{LeMS16}
O. Le M\'{e}tayer and R. Saurel, ``The Noble-Abel stiffened-gas equation of state,''
\emph{Phys. Fluids} \textbf{28}, 046102 (2016).
http://dx.doi.org/10.1063/1.4945981
\bibitem{SBLeM16}
R. Saurel, P. Boivin, and O. Le M\'{e}tayer, ``A general formulation for cavitating, boiling and evaporating flows,''
\emph{Comput. Fluids} \textbf{128}, 53--64 (2016).
http://dx.doi.org/10.1016/j.compfluid.2016.01.004
\bibitem{ZL23DM}
A. Zlotnik, ``Remarks on the model of quasi-homogeneous binary mixtures with the NASG equations of state,''
\emph{Appl. Math. Lett.} \textbf{146}, article 108801 (2023). https://doi.org/10.1016/j.aml.2023.108801

\bibitem{ZG11}
A. Zlotnik and V. Gavrilin, ``On quasi-gasdynamic system of equations with general equations of state and its application,''
\emph{Math. Model. Anal.} \textbf{16}(4), 509--526 (2011).
https://doi.org/10.3846/13926292.2011.627382

\bibitem{ZL18}
A. Zlotnik and T. Lomonosov, ``On conditions for $L^2$-dissipativity of linearized explicit QGD finite-difference schemes for one-dimensional gas dynamics equations,''
\emph{Dokl. Math.} \textbf{98}, 458--463 (2018).
https://doi.org/10.1134/S1064562418060200
\end{thebibliography}
\end{document}